\documentclass[11pt]{amsart}

\usepackage{amsmath, vmargin, times, amsfonts, amsthm, enumerate, amssymb, graphics, color, cancel, url, tikz}  
\usepackage[numbers, sort&compress]{natbib}

\newtheorem{thm}{Theorem}
\newtheorem{cor}{Corollary}
\newtheorem{lem}{Lemma}
\newtheorem{prop}{Proposition}

\newcommand{\E}[1]{\ensuremath{\mathbf{E} \left\{#1 \right \}}}
\newcommand{\Prob}[1]{\ensuremath{\mathbf{P} \left\{#1 \right \}}}
\newcommand{\N}{\mathbb N}
\newcommand{\Z}{\mathbb Z}

\newcommand{\R}{\mathbb R}

\newcommand{\Cex}{\mathcal C_{\text{ex}}}

\newcommand{\cadlag}{c{\`a}dl{\`a}g}

\newcommand*\mystrut[1]{\vrule width0pt height0pt depth#1\relax}
\author[L.~Devroye]{Luc Devroye}
\address{Luc Devroye\\
School of Computer Science\\
McGill University \\ 3480 University Street, H3A 0E9 Montreal, QC, Canada}
\email{lucdevroye@gmail.com}

\author[C.~Holmgren]{Cecilia Holmgren}
\address{Cecilia Holmgren\\
Department of Mathematics, Uppsala University, PO Box 480,
SE-751~06 Uppsala, Sweden} 
\email{cecilia.holmgren@math.uu.se}

\author[H.~Sulzbach]{Henning Sulzbach}
\address{Henning Sulzbach\\
School of Computer Science\\
McGill University \\ 3480 University Street, H3A 0E9 Montreal, QC, Canada}
\email{henning.sulzbach@gmail.com}

\title[The heavy path approach to Galton-Watson trees ]{The heavy path approach to Galton-Watson trees with an application to Apollonian networks}

\keywords{branching processes, fringe trees, spine decomposition, binary tree, continuum random tree, Brownian excursion, Apollonian networks}
\subjclass[1600]{60J80, 60J85, 05C80}      

\begin{document}
\begin{abstract}
We study the heavy path decomposition of conditional Galton-Watson trees. In a standard Galton-Watson tree conditional on its size $n$, we order all children by their subtree sizes, from large (heavy) to small. A node is marked if it is among  the $k$ heaviest nodes among its siblings. Unmarked nodes and their subtrees are removed, leaving only a tree of marked nodes, which we call the $k$-heavy tree. We study various properties of these trees, including their size and the maximal distance from any original node to the $k$-heavy tree. In particular, under some moment condition, the $2$-heavy tree is with high probability larger than $cn$ for some constant $c > 0$, and the maximal distance from the $k$-heavy tree is $O(n^{1/(k+1)})$ in probability.  As a consequence, for uniformly random Apollonian networks of size $n$, the expected size of the longest simple path is $\Omega(n)$.

\end{abstract}

\date{\today}
\maketitle

\section{Introduction}

We study Galton-Watson trees of size $n$. More precisely, we have a basic random variable $\xi$ defined by $$\Prob{\xi = i}  =p_i, \quad 0 \leq i < \infty, $$ where $(p_i)_{i \geq 0}$ is a fixed distribution. Throughout the paper, we assume that
 \begin{align*} & \sum_{i = 0}^\infty p_i = 1, \quad \E{\xi} = 1, \\
 & 0 < \sigma^2 := \E{\xi^2} - 1 < \infty. \end{align*}
The random variable $\xi$ is used to define a critical Galton-Watson process (see, e.g.\ \cite{AN}). In a standard construction, we label the nodes of the Galton-Watson tree in preorder. If $\xi_1, \xi_2, \ldots$ are independent copies of $\xi$, then node $i$ has $\xi_i$ children. Clearly, not all sequences $(\xi_i)_{1 \leq i \leq n}$ correspond to a tree of size $n$, as we must have
\begin{align*}
& (\xi_1 - 1) + \ldots + (\xi_i - 1) \geq 0 \quad \text{for all} \quad 1 \leq i < n, \\
& (\xi_1 - 1) + \ldots + (\xi_n - 1) = -1.
\end{align*}
For a given sequence $\xi_1, \xi_2, \ldots$, we define the size of a tree $\mathcal T$ as 
$$|\mathcal T| = \min \{t \geq 1: 1 +(\xi_1 - 1) + \ldots + (\xi_t - 1) = 0 \}.$$
This is a random Galton-Watson tree. Given $|\mathcal T| = n$, $\mathcal T$ is a conditional Galton-Watson tree. The family of conditional Galton-Watson trees has gained importance in the literature because it encompasses the simply-generated trees introduced by Meir and Moon \cite{meirmoon}, which are basically ordered rooted trees (of a given size) that are uniformly chosen from a class of trees. For example, when $p_0 = p_2 = 1/4, p_1 = 1/2$, the conditional Galton-Watson tree corresponds to a random binary tree of size $n$, also called a Catalan tree. When $(p_i)_{i \geq 0}$ is Poisson$(1)$, then we obtain a random labeled rooted tree, also called a Cayley tree.

\subsection{Encoding ordered rooted trees} \label{sec:encoding}

We consider two encoding functions for Galton-Watson trees of size $n$. Note that these make perfectly sense for any ordered rooted tree with $n$ nodes. For $1 \leq t \leq n$, denote by $d(t)$ the depth of the $t$-th node, where nodes are listed in preorder.
First, we define the \emph{Lukasiewicz path} $(S_i)_{0 \leq i \leq n}$ by $S_0 := 0$ and
$$S_i = (\xi_1 - 1) + \ldots + (\xi_i - 1), \quad 1 \leq i \leq n.$$
Of course, we have $S_n  = -1$ and $S_i \geq 0$ for all $0 \leq i \leq n$. Second, 
the \emph{depth-first process} (or \emph{contour function}) $(D_i)_{0 \leq i \leq 2n-2}$ is defined by $D_i = d(f(i))$, where $(f(i))_{0 \leq i \leq 2n-2}$ denotes the
node visited in the $i$-th step of the depth first traversal with respect to the preorder. Both encodings are extended to continuous functions on the respective intervals on $\R^+_0$ by linear interpolation, see Figure \ref{fig1}  below for an example.

\begin{figure}[htb] 
  \centering 
  \begin{minipage}[t]{.32\hsize} 
\begin{tikzpicture}[level distance=0.75cm,
level 1/.style={sibling distance=1.2cm},
level 2/.style={sibling distance=1cm},
level 3/.style={sibling distance=1cm},
int/.style = {circle, scale = 0.9, draw=black, fill=none, anchor = north, growth parent anchor=south},
ext/.style = {rectangle, scale = 0.7, anchor = north, growth parent anchor=south}]
\tikzstyle{every node}=[circle,draw]
\node (Root) [int] {1}
child { node (A) [int] {2} }
child { node [int] {3}
 child { node [int] {4}}
  }
child { node [int] {5}
  child { node [int] {6}}
   child { node [int] {7}}
 }
  ;
\end{tikzpicture}
\end{minipage}
\begin{minipage}[t]{.32\hsize} 
\begin{tikzpicture}[transform canvas={scale=0.5}, yscale=1.5]
\draw[line width=0.5mm,->] (0,0) -- (7.5,0) node[anchor=north west] {\begin{Huge}$t$ \end{Huge}};
\draw[line width=0.5mm,->] (0,0) -- (0,2.5) node[anchor=south east] {\begin{Huge}$S_t$ \end{Huge}};
\foreach \x in {0,1,2,3,4, 5,6,7}
    \draw[line width=0.5mm] (\x cm,2pt) -- (\x cm,-2pt);
\draw[line width=0.5mm] (1 cm,2pt) -- (1 cm,-2pt) node[anchor=north] {\begin{LARGE}$\mathbf 1$ \end{LARGE}};
\draw[line width=0.5mm] (0 cm,2pt) -- (0 cm,-2pt) node[anchor=north] {\begin{LARGE}$\mathbf 0$ \end{LARGE}};
\draw[line width=0.5mm] (7 cm,2pt) -- (7 cm,-2pt) node[anchor=north] {\begin{LARGE}$\mathbf 7$ \end{LARGE}};
\foreach \y in {0,1,2}
    \draw[line width=0.5mm] (2pt,\y cm) -- (-2pt,\y cm) node[anchor=east] {\begin{LARGE}$\mathbf \y$ \end{LARGE}};
\draw [line width=0.5mm] (0,0) -- (1,2);\draw[line width=0.5mm] (1,2) -- (2,1);\draw[line width=0.5mm] (2,1) -- (3,1); \draw[line width=0.5mm] (3,1) -- (4,0);
\draw[line width=0.5mm] (4,0) -- (5,1); \draw[line width=0.5mm] (5,1) -- (6,0); \draw[line width=0.5mm] (6,0) -- (7,-1);
\end{tikzpicture}
\end{minipage} 
\begin{minipage}[t]{.32\hsize}
\begin{tikzpicture}[transform canvas={scale=0.5}, xscale=0.5, yscale=1.5]
\draw[line width=0.5mm,->] (0,0) -- (13,0) node[anchor=north west] {\begin{Huge}$t$ \end{Huge}};
\draw[line width=0.5mm,->] (0,0) -- (0,2.5) node[anchor=south east] {\begin{Huge}$D_t$ \end{Huge}};
    \draw[line width=0.5mm] (0 cm,2pt) -- (0 cm,-2pt) node[anchor=north] {\begin{LARGE}$\mathbf 0$ \end{LARGE}};
    \draw[line width=0.5mm] (1 cm,2pt) -- (1 cm,-2pt) node[anchor=north] {\begin{LARGE}$\mathbf 1$ \end{LARGE}};
\draw[line width=0.5mm] (12 cm,2pt) -- (12 cm,-2pt) node[anchor=north] {\begin{LARGE}$\mathbf {12}$ \end{LARGE}};
\foreach \x in {0,1,2,3,4, 5,6,7, 8, 9, 10, 11, 12}
    \draw[line width=0.5mm] (\x cm,2pt) -- (\x cm,-2pt) ;
\foreach \y in {0,1,2}
    \draw[line width=0.5mm] (4pt,\y cm) -- (-4pt,\y cm) node[anchor=east] {\begin{LARGE}$\mathbf \y$ \end{LARGE}};
\draw[line width=0.5mm] (0,0) -- (1,1);\draw[line width=0.5mm] (1,1) -- (2,0);\draw[line width=0.5mm] (2,0) -- (3,1); \draw[line width=0.5mm] (3,1) -- (4,2);
\draw[line width=0.5mm] (4,2) -- (5,1); \draw[line width=0.5mm] (5,1) -- (6,0); \draw[line width=0.5mm] (6,0) -- (7,1);
\draw[line width=0.5mm] (7,1) -- (8,2);\draw[line width=0.5mm] (8,2) -- (9,1);\draw[line width=0.5mm] (9,1) -- (10,2);
\draw[line width=0.5mm] (10,2) -- (11,1); \draw[line width=0.5mm] (11,1) -- (12,0);
\end{tikzpicture}
\end{minipage}
\caption{A finite rooted tree of size $7$ with labels given by the preorder. Second and third picture show the corresponding Lukasiewicz path and depth-first-search process.}
\label{fig1}
\end{figure}
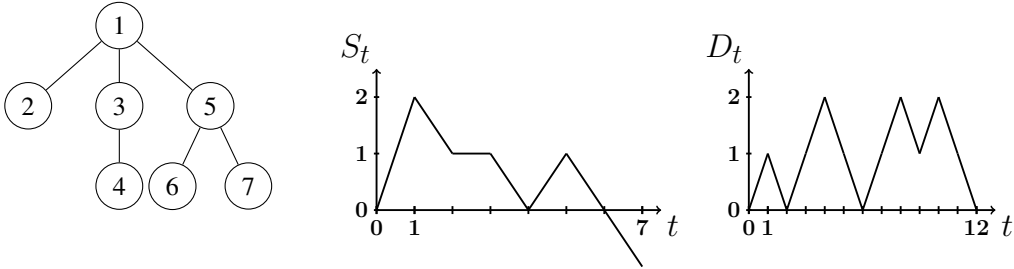

\subsection{The local and the global picture of Galton-Watson trees} \label{sec:twopic}
The history of conditional Galton-Watson trees is quite rich. Two results stand out that encapsulate our understanding:

\begin{itemize}
\item [(i)] The process 
$\left( \frac{\sigma D_{t (2n-2)}}{2 \sqrt{n}} \right)_{0 \leq t \leq 1}$ tends in distribution to a standard Brownian excursion, and the conditional tree tends in some sense to the so-called continuum random tree. This celebrated result goes back to a series of papers by Aldous \cite{al1, al2, al3}. See also Le Gall \cite{legallrandomtrees} and Marckert and Mokkadem \cite{marmok} for a discussion of convergence of both encoding processes. This implies that the height $H_n$ of the conditional Galton-Watson tree, where $H_n = \max_{1 \leq t \leq n}d(t)$, satisfies  $$\frac{\sigma H_n}{\sqrt{2n}} \stackrel{d}{\longrightarrow} H_\infty,$$ where $H_\infty$ has the theta distribution and $\stackrel{d}{\longrightarrow}$ denotes convergence in distribution. That is,  \begin{align} \label{convH} \lim_{n \to \infty} \Prob{\frac{\sigma H_n}{\sqrt{2 n}} \leq x} = \sum_{j = -\infty}^\infty (1-2j^2 x^2)\exp(-j^2 x^2), \quad x > 0. \end{align}
In this generality, this limit theorem goes back to Kolchin \cite[Theorem 2.4.3]{kolchin86}. In the case of Cayley trees, \eqref{convH} had already been discovered by R{\'e}nyi and Szerekes \cite{rensze67} and for full binary trees, that is $p_0=p_2 = 1/2$, by Flajolet and Odlyzko \cite{flod}. 

Moreover, there are universal upper bounds that will be useful for this paper: there exists $\delta \in (0, \sigma^2/2]$, such that
\begin{align} \label{boundh} \sup_{n \geq 1} \Prob{\frac{H_n}{\sqrt{n}} \geq x}  \leq \exp(-\delta x^2), \quad x > 0. \end{align}
This is Theorem 1.2 in Addario-Berry, Devroye and Janson \cite{adddevjan13}.
\item [(ii)] As $n$ grows large, $\mathcal T$ can be thought of as a long spine with offspring defined as follows: First construct an infinite sequence $\zeta_1, \zeta_2, \ldots$ drawn from the distribution $(i p_i)_{i \geq 0}$, also called the size-biased distribution. Associate $\zeta_i$ with the $i$-th node on an infinite path. To every node $i$ on the path assign $(\zeta_i - 1)$ children off the path, and make each child the root of an independent (unconditional) Galton-Watson tree. Finally permute all children of every node on the infinite spine. This infinite so-called \emph{size-biased Galton-Watson tree} is the scaling limit of conditional Galton-Watson trees as $n \to \infty$ in a much different sense than in (i). The decomposition is called the spine decomposition. The construction goes back to Kesten \cite{kesten86}. Compare also Lyons, Pemantle and Peres \cite{lpp}, Aldous and Pitman \cite[Section 2.5]{alpi98} and Janson \cite[Section 7]{janson2012a}.
\end{itemize}

Let us stress that the two pictures drawn focus on very different aspects of the  trees. Aldous' theory leading to the continuum random tree describes the global structure and  is useful in the analysis of the tree height, diameter or the depth of a uniformly chosen node. 
The convergence result constitutes an invariance principle: in analogy to the central limit theorem for independent and identically distributed summands, it only relies on the second moment of the offspring distribution. This also means that more local information, that is,  quantities which scale on order smaller than $\sqrt n$, cannot be studied by this method.

The picture drawn in (ii) is local: the conditional Galton-Watson tree converges locally, in the sense of Aldous-Steele \cite{ob} (sometimes also referred to as Benjamini-Schramm convergence \cite{benjschramm}), to the infinite size-biased tree. To be more precise, it states that, for any fixed $k \geq 1$, the conditional Galton-Watson tree restricted to nodes of distance at most $k$ from the root converges as $n \to \infty$ in distribution to the restricted object sampled from the infinite size-biased tree. 

\medskip The present paper looks at a less natural decomposition of the conditional Galton-Watson tree,  but one that has far-reaching applications in computer science and the study of random networks, more precisely, random Apollonian networks. 

\subsection{Heavy subtrees and main results} \label{sec:mainresults} One can reorder all sets of siblings by subtree size, from large to small, where ties are broken by considering the preorder index. For a node $v$ in the (conditional or not) Galton-Watson tree $\mathcal T$, we denote by $\rho_v$ the rank in its ordering (for example, $\rho_v = 1$ means that $v$ has the largest subtree among its siblings). Let $A_v = (v_1, \ldots, v_{d-1} = v)$ be the sequence of ancestors of $v$, starting at the root and ending at $v$ if $v$ is at distance $d-1$ from the root. We define the maximal rank $$\rho_v^* = \max (\rho_{v_1}, \ldots, \rho_{v_{d-1}}).$$ No rank is defined for the root.  For fixed integer $k$, we define
the \emph{$k$-heavy Galton-Watson tree}  as the tree formed by the root and $\{v \in \mathcal  T : \rho_v^* \leq k\}$, where $\mathcal T$ is the conditional Galton-Watson tree. The $k$-heavy tree has nodes of degree $k$ or less. For $k = 1$, we obtain a path, which we call the \emph{heavy path}---just follow the path from the root down, always going to the largest subtree. It is interesting that the length $L_n$ of the heavy path has a different asymptotic distributional behaviour than $H_n$.  Clearly, $L_n \leq H_n$, but $L_n$ is neither too small nor too close to $H_n$. In Section \ref{sec:heavypath} we discuss distributional convergence of $L_n / \sqrt{n}$ and study the tail behaviour of the random variable $L_n / \sqrt{n}$ near $0$ in more detail. We note that it grows more slowly than any polynomial but much faster than the theta law (see \eqref{convH}). As opposed to the $k$-heavy trees for $k \geq 2$, the heavy path can be studied using the global picture (i) sketched above, and its scaling limit has a representation in terms of a Brownian excursion (or the continuum random tree).

\medskip Our main interest, though, is the study of the case $k = 2$, the $2$-heavy Galton Watson tree. In Section \ref{sec:2ary}, we show that it captures a huge chunk of the Galton-Watson tree: by Theorem \ref{thm:low2ary}, if $\E{\xi^5} < \infty$, then there exists a constant $c > 0$ such that
\begin{align} \label{bin} \lim_{n \to \infty} \Prob{ \text{Size of the} \ 2  \text{-heavy tree} \geq c n} = 1. \end{align}
Since the number of nodes of degree $i$ in a conditional Galton-Watson tree is in probability asymptotic to $n p_i$, it is easy to see that the size of the $2$-heavy tree cannot be more than $$n \left(1 - \sum_{i \geq 3} (i-2) p_i + o(1)\right),$$
so that there is no hope of replacing $cn$ by $n - o(n)$ in \eqref{bin}. In fact, we believe that the size of the $2$-heavy tree satisfies a law of large numbers when rescaled by $n^{-1}$ as $n \to \infty$ with a limiting constant depending on the distribution of $\xi$.

Finally, we also study the maximal distance to the $k$-heavy trees. For a proper set of nodes, $A \subseteq \{1, \ldots, n\}$, we call the maximal distance to $A$
$$\max_{v \notin A} \min _{w \in A} \text{dist}(v,w), $$
where $\text{dist}(\cdot, \cdot)$ refers to path distance. The maximal distance to the $k$-heavy tree measures to some extent how pervasive the $k$-heavy tree is. In Section \ref{sec:distances}, we show that, under appropriate moment conditions on $\xi$, the distance is in probability $\Theta(n^{1/(k+1)})$. In fact, we also show that this is optimal in the sense that, every $k$-ary subtree leaves out nodes of distance order $n^{1/(k+1)}$ away.

\subsection{Apollonian networks} In 1930, Birkhoff \cite{birk} introduced a model of a planar graph that became known as an Apollonian network, a name coined by Andrade et al.\ \cite{andrade} in 2005. Suggested as toy models of social and physical networks with remarkable properties, they are recursively defined by starting with three vertices that form a triangle in the plane. Given a collection of triangles in a triangulation, choose one (either at random, or following an algorithm), place a new vertex in its center, and connect it with the three vertices of the triangle. So, in each step, we create three new edges, one new point, and three new triangles (which replace an old one). After $n$ steps, we have $3 + n$ vertices, and $3 + 3n$ edges in the graph. This is an Apollonian network. One can also define a dual tree: start with the original triangle as the root of a tree. In a typical step, select a leaf node of the tree (which corresponds to a triangle) and attach to it three children. This tree has a one-to-one relationship with the Apollonian network.  It has $1 + 2n$ leaves (after $n$ steps) and $1 + 3n$ vertices. See Figure \ref{fig:apo} for an illustration. 

        \begin{figure}[h]  \centering
\includegraphics[width=.6\columnwidth]{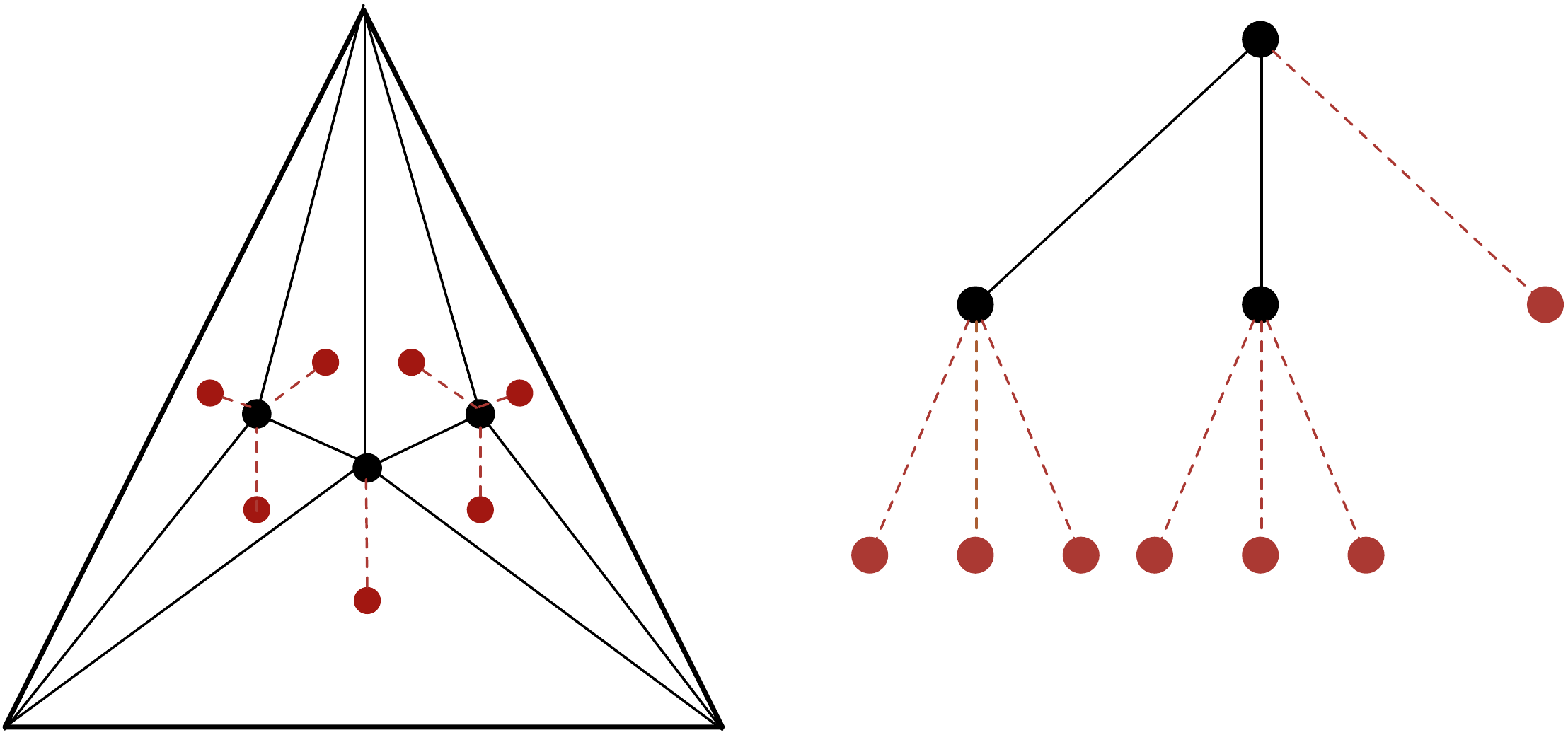} 
\caption{Apollonian network of size 3 with evolutionary tree. Leaves are drawn in red. }  \label{fig:apo} \end{figure}

A frequently studied (see Zhou, Yan and Wang \cite{zhou}) random Apollonian network is one in which each triangle (in the network)---or, equivalently, each leaf in the tree---is chosen uniformly at random for splitting, leading to a so-called split tree \cite{dev_split}. (More precisely, we obtain a random ternary increasing tree, a variant of the much studied random binary search tree.) Its height is bounded almost surely by $c \log n$ for a suitable constant $c > 0$ \cite{nic_dev}. More importantly, one is interested in the longest simple path in the Apollonian network. (A simple path in a graph is a path which visits every vertex at most once.) Calling its length $\mathcal L_n$, its asymptotic behaviour is still not well understood today. 
Takeo \cite{takeo} erroneously claimed that Apollonian networks have a Hamiltonian cycle (and thus, $\mathcal L_n = n -1$), but the so-called Goldner-Harary graphs invented by Gr{\"u}nbaum in 1967 \cite{grun} form just one of many possible counterexamples. Frieze and Tsourakakis \cite{frt} conjectured that for the random Apollonian network of Zhou, Yan and Wang,  $\mathcal L_n \geq cn$ for some constant $c > 0$ with probability tending to one. This was disproved by Ebrahimzadeh et al. \cite{ebra1} who showed that, with high probability, $\mathcal L_n = o(n)$. They also provided a lower bound of $\Omega(n^{0.88})$ for $\E{\mathcal L_n}$. 
Very recently,  Collevecchio, Mehrabian and Wormald \cite{cmw} proved that $\mathcal L_n$ is with high probability at most 
$n^{1-\varepsilon}$ where $\varepsilon$ can be chosen $4 \times 10^{-8}$.

If the random model is changed, and we generate a random ordered tree of size $1 + 3n$ in which each non-leaf node has three children, such that all trees are equally likely, then this corresponds to a conditional Galton-Watson tree (of size $1 + 3n$) with $p_0 = 2/3, p_1 = p_2 = 0$ and $p_3 = 1/3$. Furthermore, it is easy to verify that the length of the longest simple path $\mathcal L_n$ is bounded from below by the size of any binary subtree embedded in the Galton-Watson tree. In particular, it is larger than the size of the $2$-heavy tree. Therefore, there exists $c > 0$ such that \begin{align} \label{conjfrieze} \lim_{n \to \infty} \Prob{\mathcal L_n \geq cn} = 1.\end{align}
Thus, for this random model, Frieze and Tsourakakis' conjecture is easily settled by studying the $2$-heavy tree. This was the initial motivation of the present paper.

\subsection{Notation} Throughout the paper, we use  $h = \gcd \{ i : p_i > 0, i > 0\},$ $\alpha = h / (\sigma  \sqrt{2\pi})$,  $I = \{n \geq 1: \Prob{S_n = -1} >0 \}$, and, for $n \in I$, 
$I_n = \{1 \leq k \leq n: \Prob{S_{n-k} = 0} >0 \}$.
From B{\'e}zout's lemma, it follows that $I = (\N h  + 1)\backslash A$ for some finite set $A \subseteq \N$. 
In the remainder of the paper, we write $\mathcal T$ for a realization of the unconditional Galton-Watson tree and $\tau_n, n \in I,$ for $\mathcal T$ conditional on having size $n$. ($\mathcal T$ and $\mathcal \tau_n$ are considered as graphs, where $\tau_n$ has vertex set $[n] := \{1, \ldots, n\}$.) We introduce the following terminology: for $v \in [n]$, let
 $\xi(v)$ be the number of children of $v$,
$N(v)$ be the size of the subtree rooted at $v$,
$H(v)$ be the height of the subtree rooted at $v$,
$N_i(v)$ be the size of the $i$-th largest subtree rooted at $v$, abbreviating $N_i(v) = 0$ if $i > \xi(v)$, and 
 $N_{i+}(v) = N_i(v) + N_{i+1}(v) + \ldots$
We use the notation $\xi_\epsilon, N, H, N_i$ and  $N_{i+}$ when referring to the root node. Finally,  for $k \geq 1$, let $Z_k = | \{v \in [n]: N(v) = k\}|$.  We stress that, in order to increase readability, we often omit to indicate the parameter $n$ in the notation.

In Sections \ref{sec:fringe} -- \ref{apa}, Appendices A and B, all constants except $c, c_1, c_2, \ldots, C, C_1, C_2, \ldots$ carry fixed values. The values of constants used multiple times may vary between two results or proofs but not within. Here, constants $C, C_1, C_2, \ldots > 0$ are meant to carry \emph{large} values, whereas $c, c_1, c_2, \ldots > 0$ are typically \emph{small}. Appendix D can be read independently of the remainder of the work.

\section{Preliminary results and fringe trees} \label{sec:fringe}
Let us start by recovering some classical results which have proved fruitful in the analysis of conditional Galton-Watson trees. 
Recall the following well-known identity going back to Dwass \cite{dwass69} (compare also Janson \cite[Theorem 15.5]{janson2012a} and the discussion therein),
\begin{align} \label{ident1}
\Prob{|\mathcal T| = n} = \frac{\Prob{S_n = -1}}{n}.
\end{align}
More generally, for independent copies  $\mathcal T_1, \mathcal T_2, \ldots$ of $\mathcal T$,
\begin{align} \label{ident22}\Prob{|\mathcal T_1| + \ldots + |\mathcal T_k| = n} = \frac k n \Prob{S_n = -k}. \end{align}
In this context, we cite a classical result for sums of independent integer random variables applied to the sequence $S_n$. By Petrov \cite[Theorem VII.1]{petrov75} or Kolchin \cite[Theorem 1.4.2] {kolchin86}, as $n \to \infty$, 
\begin{align} \label{LLTgen} \sup_{x \in \Z h - n} \left|\Prob{ S_n = x} - \frac{\alpha}{\sqrt{n}} \exp \left( - \frac{x^2}{2 \sigma^2 n}\right) \right| = o(n^{-1/2}). \end{align}
In particular, for $x = o(\sqrt{n})$ with $x \in \Z h - n $, as $n \to \infty$, 
\begin{align*}  \Prob{S_n = x} \sim \frac{\alpha}{ \sqrt{n }}. \end{align*}
Similarly, as $n \to \infty, n \in \N h + 1$, 
\begin{align} \label{LLT}
\Prob{S_n = -1} \sim \frac{\alpha}{ \sqrt{n }}.
\end{align}
By summation, using \eqref{ident1} and \eqref{LLT}, as $t \to \infty$,
\begin{align} \label{tailT}
\Prob{|\mathcal T| \geq t} \sim \frac{2 \alpha}{h \sqrt{t}}.
\end{align}

The study of the sequence $Z_k, k \geq 1$ is closely related to the analysis of a random fringe subtree $\tau^*_n$, a subtree of $\tau_n$ rooted at a uniformly chosen node. The study of fringe subtrees was initiated by Aldous \cite{aldous91}, who showed that, under our conditions,
\begin{align} \label{aldousfringe} \tau^*_n \stackrel{d}{\longrightarrow} \mathcal T. \end{align}
In particular, $\E{Z_k}/n \to \Prob{|\mathcal T| = k}$ as $n \to \infty, n \in \N h +1$ for $k \in I$ fixed. Bennies and Kersting \cite[Theorem 1]{benker00} generalized \eqref{aldousfringe} to offspring distributions with infinite variance. Janson \cite[Theorem 7.12]{janson2012a} obtained a quenched version of the result stating that, conditional on the tree $\tau_n$, the (random) distribution of $\tau_n^*$ converges in probability to the (deterministic) distribution of $\mathcal T$.
More recently, Janson \cite[Theorem 1.5]{janson2013} obtained finer results on subtree counts in $\tau_n$, in particular estimates and asymptotic expansions for the variance and a central limit theorem. We summarize special cases of his results in the following proposition. The exact expressions for mean and variance are contained in \cite[Lemma 5.1]{janson2013} and \cite[Lemma 6.1]{janson2013}. The uniform estimate on the variance \eqref{variance} follows from \cite[Theorem 6.7]{janson2013}.

\begin{prop}[Janson \cite{janson2013}]
Let $n \in I$ and $1 \leq k \leq n$. Then, 
\begin{align} \label{exactmean}
\E{Z_k} = \frac{n \Prob{S_k = -1} \Prob{S_{n-k} = 0}}{k \Prob{S_n = -1}}, 
\end{align}
and, for $1 \leq k \leq (n-1)/2$, 
\begin{align*} \E{Z_k (Z_k - 1)} & = \frac{n(n-2k+1)\Prob{S_k = -1}^2 \Prob{S_{n-2k} = 1}}{k^2 \Prob{S_n = -1}},  
\end{align*}
while $\E{Z_k (Z_k - 1)} = 0$ for $k > (n-1)/2$.
For fixed $k \in I$, as $n \to \infty, n \in \N h +1$, 
\begin{align*}  \frac{\emph{Var} (Z_k)}{n} \to \theta^2, \quad  \theta^2 = \Prob{|\mathcal T| = k} \left[1 + \Prob{|\mathcal T| = k} (1-2k - \sigma^{-2}) \right] > 0. \end{align*}
and,  
\begin{align*} \frac{Z_k - n \Prob{|\mathcal T| = k}}{\sqrt{n}} \stackrel{d}{\longrightarrow} \mathcal N(0, \theta^2),
\end{align*}
where $ \mathcal N(0, \theta^2)$ denotes a normal random variable with variance $\theta^2$ and mean $0$.

Finally, uniformly in $1 \leq k \leq n$, as $n \to \infty, n \in \N h +1$, 
\begin{align} 
\label{variance}
\emph{Var}(Z_k) = O\left(n\right).
\end{align}
\end{prop}

It follows from \eqref{LLT} that, as $n \to \infty$ and $k = o(n)$ with $n \in \N h + 1$, $k \in I_n \cap I$, in probability,
\begin{align*}  
\frac{Z_k}{\E{Z_k}} \to 1. \end{align*}
Furthermore, if additionally $k \to \infty$, then
\begin{align} \label{asy:mean} \E{Z_k} \sim  \frac{\alpha n}{ k ^{3/2}}. \end{align}

Many arguments in this manuscript rely on bounds on the mean such as those given below.
\begin{cor} \label{bounds:mean}
There exists a constant $n_0 \geq 1$ such that, for all $n \geq n_0$, $n \in I$, $k \in I_n$, 
\begin{align*} \E{Z_k} \leq \begin{cases} 4 \sqrt{2} \alpha (n-k)^{-1/2}  &\mbox{for } n/2 \leq k \leq n - n_0, \\ 
2 \sqrt{2} \alpha n k^{-3/2} & \mbox{for } n_0 \leq k \leq n/2, \\ \lfloor n/k \rfloor & \mbox{for } 1 \leq k \leq n. \end{cases} \end{align*}
Similarly, there  exist constants $n_1 \geq 1$ and $ \varsigma >0$, such that, for all $n \geq n_1$, $n \in I$, $k \in I_n \cap I$, 
\begin{align*} \E{Z_k} \geq \begin{cases} 
\alpha n k^{-3/2}/2 & \mbox{for } n_1 \leq k \leq n/2, \\   \varsigma n & \mbox{for } 1 \leq k \leq n_1. \end{cases} \end{align*}
\end{cor}
\begin{proof}
By an application of \eqref{LLT} to \eqref{exactmean}, there exists $n_0 \geq 1$, such that, for all $n_0 \leq k \leq n-n_0$, 
\begin{align*}
\frac{\E{Z_k}}{n} & \leq 2 \alpha k ^{-3/2} \sqrt{\frac{n}{n-k}} \leq \begin{cases}  2 \sqrt{2} \alpha k ^{-3/2}  &\mbox{for } n_0 \leq k \leq n/2, \\ 
 \frac{4 \sqrt{2} \alpha }{n \sqrt{n-k}} & \mbox{for } n/2 \leq i \leq n-n_0. \end{cases} \end{align*}
This shows the first two upper bounds. The third follows immediately from the deterministic bound $Z_k \leq \lfloor n / k \rfloor$. The first lower bound follows analogously. The second lower bound follows from \eqref{ident1} and \eqref{asy:mean}, since, for $k \in I$, we have $\E{Z_k} / n \to \Prob{|\mathcal T| = k} = \Prob{S_k = -1} / k$. 
\end{proof}
\begin{cor} \label{logbound}
There exists a universal constant $C > 0$ such that, for all $M \geq n_0$ and $n \geq M, n \in I,$ with $n_0$ as in Corollary \ref{bounds:mean}, we have
\begin{align*}
\sum_{k=M}^n \E{Z_k} \log k \leq C n \frac{\log M}{\sqrt{M}} + (2 + 4 \sqrt{2} \alpha)  n ^{3/4} \log n.
\end{align*}
\end{cor}
\begin{proof} 
By applications of the upper bound in the previous theorem, we have
\begin{align*}
& \sum_{k = M}^{\lfloor n/2 \rfloor} \E{Z_k} \log k \leq  2 \sqrt{2} \alpha n \sum_{k = M}^{\lfloor n/2 \rfloor}  k^{-3/2} \log k \leq C n \frac{\log M}{\sqrt{M}}, \\
& \sum_{k = \lceil n/2 \rceil}^{\lfloor n-\sqrt{n}\rfloor } \E{Z_k} \log k \leq \frac{\log n}{\sqrt{n-n+\sqrt{n}}}  \sum_{k = \lceil n/2 \rceil}^{\lfloor n-\sqrt{n}\rfloor } 4 \sqrt{2} \alpha \leq 4 \sqrt{2} \alpha n^{3/4} \log n, \\
&  \sum_{k = \lceil n - \sqrt{n} \rceil }^{n} \E{Z_k} \log k \leq n \log n \sum_{k = \lceil n - \sqrt{n} \rceil }^{n}  \frac 1 k \leq 2 \sqrt{n} \log n.
\end{align*}
The claim follows by summing the three terms.
\end{proof}

\section{Subtrees of the root: local convergence}
We want to understand the properties of the subtree sizes of a node in a Galton-Watson tree conditional on having size $n$ when these trees are ordered from large to small. This section has key  inequalities that will be needed throughout the paper.

A formulation of the local convergence result discussed in Section \ref{sec:twopic} (ii) is given in the next proposition which is equivalent to Lemma 1 in Devroye \cite{devroye11}. (The convergence of $\xi_\epsilon$ had already been obtained by Kennedy \cite{kennedy75}.) We include the short proof for the sake of completeness. Here, by $S_\downarrow$, we denote the set of non-negative integer valued sequences $x_1, x_2, \ldots$ with $x_1 \geq x_2 \geq \ldots$ and only finitely many non-zero elements. Note that $S_\downarrow$ is countable. For $k \geq 1$ and $1 \leq i \leq k$, and real-valued random variables $X_1, \ldots, X_k$, denote by $X_{(i:k)}$ the $(k-i+1)$-st order statistic. (For random trees $\mathcal T_1, \ldots, \mathcal T_k$, we simplify the notation and write $|\mathcal T_{(i:k)}|$ for the size of $i$-th largest tree.)

\begin{prop} \label{prop2}
Let $\zeta$ have the size-biased distribution $(ip_i)_{ i\geq 0}$. Then, as $n \to \infty$, $n \in \N h + 1$,  in distribution on $S_\downarrow$, $$(N_2, N_3, \ldots) \to (|\mathcal T_{(1:\zeta -1)}|, \ldots, |\mathcal T_{(\zeta -1:\zeta -1)}|, 0, 0, \ldots),$$ where $\mathcal T_1$, $\mathcal T_2, \ldots,$ $\zeta$ are independent. In distribution and in mean, $\xi_\epsilon \to \zeta$, where we recall that $\xi_\epsilon$ is the number of children of the root of $\tau_n$. The convergence is with respect to the $k$-th moment if and only if $\E{\xi^{k+1}} < \infty$. 
\end{prop}
\begin{proof}
Let $k_1, k_2, \ldots \in S_\downarrow$ with $k_i \in \N h, 1 \leq i \leq \ell - 1$, $k_\ell > 0, k_{\ell +1} = k_{\ell + 2} = \ldots = 0$ and $p_{\ell + 1} > 0$. Let $y_1, \ldots, y_m$ be the different values among $k_1, \ldots, k_\ell$ and $\alpha_1, \ldots, \alpha_m$ be their multiplicities. 
With $C = {\ell \choose \alpha_1, \ldots, \alpha_m}$, 
\begin{align*}
\Prob{(|\mathcal T_{(1:\zeta-1)}|, \ldots, |\mathcal T_{(\zeta -1:\zeta-1)}|, 0, 0, \ldots) = (k_1, k_2, \ldots)} = C p_{\ell +1} (\ell + 1) \prod_{i=1}^\ell \Prob{|\mathcal T| = k_i}.
\end{align*}
Similarly, for all $n \in I_n$ with $n > 1 + k_1 + \sum_{i=1}^\ell k_i$, 
\begin{align*}
\Prob{(N_2, N_3, \ldots) = (k_1, k_2, \ldots)} =C p_{\ell +1} (\ell + 1) \prod_{i=1}^\ell \Prob{|\mathcal T| = k_i} \frac{\Prob{|\mathcal T| = n-1-\sum_{j=1}^\ell k_j}}{\Prob{|\mathcal T| = n}}.
\end{align*}
The distributional convergence in $S_\downarrow$ follows since the fraction in the last display turns to one as $n \to \infty$.
Since $S_\downarrow$ is countable, the function $f : S_\downarrow \to \N, f(x_1, x_2, \ldots ) = \min \{k \geq 1: x_k  = 0\}$ is continuous, and we deduce 
$\xi_\epsilon \to \zeta$ in distribution. Furthermore, for $k \in \N h$, using \eqref{ident1} and \eqref{ident22},
\begin{align*}
\Prob{\xi_\epsilon = k} & = p_k \frac{\Prob{|\mathcal T_1| + \ldots + |\mathcal T_k| = n-1}}{\Prob{|\mathcal T| = n}} \\
& = k p_k \left(1 - \frac 1 n\right) \frac{\Prob{S_{n-1} = -k}}{\Prob{S_n = -1}}.
\end{align*}
Since the fraction is uniformly bounded in $k$, $\xi_\epsilon^k$ is uniformly integrable if $\zeta^k$ is integrable. Finally, if $\sum_{\ell \geq 1} \ell^{k+1} p_\ell = \infty$, then $\E{\xi_\epsilon^k} \to \infty$ by Fatou's lemma. This concludes the proof. 
\end{proof}

We are interested in tail bounds on $N_k, k \geq 2$. The order is suggested by the behavior of the limiting random variable. 
\begin{prop}
Let $k \geq 1$ and assume that $\mathcal T_1$, $\mathcal T_2, \ldots,$ $\zeta$ are independent. \begin{itemize} 
\item [(i)] If $\E{\xi^{k+1}} < \infty$, then, as $t\to \infty$, 
\begin{align} \label{tail:limit} \Prob{|\mathcal T_{(k: \zeta-1)}| \geq t} = O(t^{-k/2}). \end{align} 
\item [(ii)] If $\sum_{\ell \geq k+1} p_\ell > 0$, then, as $t\to \infty$, 
\begin{align*}  \Prob{|\mathcal T_{(k: \zeta-1)}| \geq t} = \Omega(t^{-k/2}). \end{align*}
\item [(iii)] Finally, if $\E{\xi^{k+1}} = \infty$, then
$$\lim_{t\to \infty} t^{k/2} \Prob{|\mathcal T_{(k: \zeta-1)}| \geq t} = \infty.$$
\end{itemize}
\end{prop}
\begin{proof}
We have 
\begin{align*}
\Prob{|\mathcal T_{(k: \zeta-1)}| \geq t} & \leq \sum_{\ell \geq k} p_{\ell + 1}(\ell + 1) {\ell \choose k} \Prob{|\mathcal T_{1}| \geq t, \ldots, |\mathcal T_{k}| \geq t}.
\end{align*}
By \eqref{tailT}, the right-hand side is asymptotically equivalent to
\begin{align*}
& \left(\frac{2 \alpha}{h \sqrt{t}}\right)^k  \sum_{\ell \geq k} p_{\ell + 1}(\ell + 1) {\ell \choose k}.
\end{align*}
Since $\E{\xi^{k+1}} < \infty$, the term is of order $t^{-k/2}$. For (ii), choose $\ell \geq k$ with $p_{\ell+1} > 0$. Then, 
\begin{align*}
\Prob{|\mathcal T_{(k: \zeta-1)}| \geq t} & \geq  p_{\ell + 1}(\ell + 1) \Prob{|\mathcal T_{1}| \geq t, \ldots, |\mathcal T_{k}| \geq t}  \sim \left(\frac{2 \alpha}{h \sqrt{t}}\right)^k \ p_{\ell + 1}(\ell + 1).
\end{align*}
Again, the right hand side is of order $t^{-k/2}$. 

For (iii), since $\E{\xi^{k+1}} = \infty$, for any $C > 0$, find $K$ sufficiently large such that $\sum_{\ell =k}^K p_\ell (\ell+1) {\ell \choose k} \geq C$. Then
\begin{align*}
\Prob{|\mathcal T_{(k: \zeta-1)}| \geq t} & \geq \Prob{|\mathcal T_{(k: \zeta-1)}| \geq t, |\mathcal T_{(k+1: \zeta-1)}| < t} \\
& = \sum_{\ell \geq k} p_{\ell + 1}(\ell + 1) {\ell \choose k} \Prob{|\mathcal T| \geq t}^k \Prob{|\mathcal T| < t}^{\ell-k}  \\
& \geq C \Prob{|\mathcal T| \geq t}^k \Prob{|\mathcal T| < t}^K. 
\end{align*}
As $t \to \infty$, using \eqref{tailT}, the right hand side is equivalent to $C (2\alpha h^{-1})^k t^{-k/2}$. As $C$ was chosen arbitrarily, the final assertion of the proposition follows.
\end{proof}

The next two results are proved in Appendix A.

\begin{thm} \label{thm:tailbounds}
Let $k \geq 2$ and $\E{\xi^{k+1}} < \infty$. Then, there exists a constant $\beta_k >0$, such that, for all $t \geq 1,  n \in I$, 
\begin{align} \label{firsttail}
\Prob{N_k \geq t} \leq \beta_k t^{(1-k)/2}.
\end{align}
If $\E{\xi^{(3k+1)/2}} < \infty$, a corresponding bound holds for $\Prob{N_{k+} \geq t}$ with $\beta_k$ replaced by some larger constant $\beta_{k+}$.
Similarly, bounds of the same form are valid for $\E{\xi_\epsilon \mathbf{1}_{\{N_k \geq t\}}}$ if $\E{\xi^{k+2}} < \infty$, and for $\E{\xi_\epsilon \mathbf{1}_{\{N_{k+} \geq t\}}}$ if $\E{\xi^{(3k+3)/2}} < \infty$.
\end{thm}

\textbf{Remark.} The proof of Theorem \ref{thm:tailbounds} shows the following stronger result: for $k \geq 2$, there exists a constant $C >0$ such that, for all $n \in I, \ell \geq k$ and $t \geq 1$, 
\begin{align} \label{tailrem}
\Prob{N_k \geq t, \xi_\epsilon = \ell} \leq C p_\ell \ell^{k+1} t^{(1-k)/2}.
\end{align}
Lemma \ref{lem:conheight} below is the only result in this work that requires this stronger bound. 

\medskip \textbf{Remark.} Since $N_k \stackrel{d}{\longrightarrow} \mathcal T_{(k-1 : \zeta - 1)}$ and the moment condition on this random variable in order to have tails decaying as in  \eqref{tail:limit} is tight, it is reasonable to conjecture that a tail bound such as \eqref{firsttail} holds if and only if $\E{\xi^k} < \infty$. The bounds presented in Appendix A are sufficient to show that the latter is indeed necessary: if $\E{\xi^k} = \infty$, then a bound of the form \eqref{firsttail} is not valid. (A proof of this claim is given in Appendix A.)

\begin{thm}  \label{thm:tailbounds2}
Let $k \geq 2$ and $\sum_{\ell \geq k} p_\ell > 0$. Then, there exist constants $\beta_k^* > 0, s_k > 0$ and $n_2 = n_2(k) \geq 1$, such that, for all $n \geq n_2, n \in I,$ and $1 \leq t \leq n / k - s_k$, 
\begin{align*}
\Prob{N_k \geq t} \geq \beta_k^* t^{(1-k)/2}.
\end{align*}
\end{thm}

From Theorems \ref{thm:tailbounds} and  \ref{thm:tailbounds2} we deduce the following corollary using the  well-known formula $\E{X} = \int_0^\infty \Prob{X > t} dt$ for a non-negative random variable $X$.

\begin{cor} \label{subtreebounds}
As $n \to \infty$, $n \in \N h + 1$, 
\begin{itemize}
\item [(i)] if $\E{\xi^3} < \infty$, then $\E{N_2} = \Theta(\sqrt{n})$ and $\E{\sqrt{N_2}} = \Theta(\log{n})$, 
\item [(ii)] if $\E{\xi^{7/2}} < \infty$, then $\E{N_{2+}} = \Theta(\sqrt{n})$,
\item [(iii)] if $\E{\xi^4} < \infty$, then $\E{N_3} = O(\log n)$, 
\item [(iv)] if $\E{\xi^5} < \infty$, then $\E{N_{3+}} =  O(\log n)$ and $\E{N_4} = O(1)$, and
\item [(v)] if $\E{\xi^{13/2}} < \infty$, then $\E{N_{4+}} = O(1).$ 
\end{itemize}
If $\sum_{\ell \geq 3} p_\ell > 0$, then big-$O$ in (iii) can be replaced by $\Theta$.
\end{cor}

\section{The $2$-heavy tree} \label{sec:2ary}
Let $\mathbb T$ be a finite ordered rooted tree with vertex set $\mathcal V(\mathbb T)$. Its root is labeled $\epsilon$. As in Section \ref{sec:mainresults}, to each node $v \in \mathcal V(\mathbb T)$, $v \neq \epsilon$, we assign the rank $\rho_v$ where $\rho_v = i$ if its subtree is the $i$-th largest among all the subtrees rooted at its siblings. Ties are broken by the original order in the tree. 
If $v$ has  distance $k \geq 1$ from $\epsilon$, let $v_0 := \epsilon, v_1, \ldots,$ $v_{k-1}, v_{k} = v$ be the nodes on the path connecting the root to $v$ where $v_i$ has depth $i$. The path from $\epsilon$ to $v$ has nodes of indices $\rho_{v_1}, \ldots, \rho_{v_k} = \rho_v$. It is called the index sequence of $v$ and denoted by $\kappa(v)$. We define $\kappa(\epsilon) = \emptyset$ as the empty word. It is convenient to borrow some notation from theoretical computer science for sequences of integers: $\{i_1, \ldots, i_k\}$ denotes one symbol from the set $\{i_1, \ldots, i_k\}$ and $A^*$ denotes a sequence of arbitrary length (even $0$) drawn from $A \subseteq \N$. We define the set of nodes $\mathcal V$ satisfying a sequence as the collection of all nodes in the tree have index sequences belonging to a set of sequences. For example, $\mathcal V(1^*) := \mathcal V(\{1\}^*)$ is the set of nodes in $\mathbb T$ that have all their ancestors and itself of index $1$ and the root. Of course,  the nodes in $\mathcal V(1^*)$ form the heavy path.
Furthermore, we recover the $k$-heavy tree $\mathcal V(\{1, \ldots, k \}^*)$ of $\mathbb T$ by removing from $\mathbb T$ all nodes of index larger than $k$ and its subtrees. For $k = 2$, we obtain the $2$-heavy tree. The $2$-heavy Galton-Watson tree is denoted by $\mathcal B_n$, and its size  by $B_n$. It is tempting to think that $B_n$ is increasing in probability or, at least, in mean. The following example shows that this is not the case. Let $p_0, p_2, p_5 > 0$ with $p_0 + p_2 + p_5 = 1$. Then, on the one hand, almost surely, $\tau_5$ is binary and $B_5 = 5$. On the other hand, almost surely, $\tau_6$ consists of the root with five children. Thus $B_6 = 3$. Note that this issue can not be avoided by assuming $p_i > 0$ for all $i$.
\begin{thm} \label{thm:meanbinary}
Let $\E{\xi^5} < \infty$. There exist constants $\nu_1, \nu_2 > 0$, such that, for all $n \in I$, 
\begin{align} \label{lowbound}
\E{B_n} \geq \nu_1 n  + \nu_2 \sqrt{n} - \frac 1 2.
\end{align}
\end{thm}
\begin{proof} The proof uses induction. First, since $B_i = i$ for $i \in \{0, 1,2, 3\} \cap I$, we need to have 
\begin{align} \label{0b}
 \nu_1 + \nu_2  \leq 3/2, \quad 2\nu_1 + \sqrt{2} \nu_2  \leq 5/2, \quad 3\nu_1 + \sqrt{3} \nu_2 \leq 7/2.\end{align}
Assume \eqref{lowbound} holds up to $n-1$ (in the set $I$) with $n \geq 4$. Then, for $n \in I$, and with $b(n) = \E{B_n}$, 
\begin{align*}
b_n & = 1 + \E{b(N_1)} + \E{b(N_2)} \\
& \geq  \nu_1 \E{N_1 + N_2} + \nu_2 \E{\sqrt{N_1} + \sqrt{N_2}}  \\
& =  \nu_1(n-1) - \nu_1 \E{N_{3+}} + \nu_2 \E{\sqrt{n-1 - N_{2+}} + \sqrt{N_2}} \\
& \geq  \nu_1(n-1) - \nu_1 \E{N_{3+}} + \nu_2 \sqrt n - \nu_2 - \nu_2  \E{N_{2+}} / \sqrt{n-1} + \sqrt{N_2}.
\end{align*}
Here, in the last step, we have used that $1 - x \leq \sqrt{1-x}$ for all $x \in [0,1]$. By the previous corollary, there exist strictly positive constants $C_1, c_2, C_3$, such that 
\begin{align*}
\E{N_{3+}} \leq C_1 \log n, \quad \E{\sqrt{N_{2}}} \geq c_2 \log n , \quad \E{N_{2+}} \leq C_3 \sqrt{n-1}. 
\end{align*}
Thus,
\begin{align*} b_n \geq  - \nu_1 - \nu_2 - \nu_2 C_3 + (\nu_2 c_2 - \nu_1 C_1) \log n + \nu_1 n + \nu_2 \sqrt{n}.
\end{align*}
From here, the claim $b_n \geq \nu_1n + \nu_2 \sqrt{n} -1/2 $ follows if both
\begin{align*}
1/2 - \nu_1 - \nu_2(1+C_3)  \geq 0, \quad \text{and},  \quad \nu_2 c_2 - \nu_1 C_1 \geq 0. \end{align*}
The last expression and all inequalities in \eqref{0b} can simultaneously be satisfied by choosing 
 $\nu_2 = \nu_1 C_1 /c_2$ and 
\begin{align*}
\nu_1 \leq \min \left \{ \frac {3}{2(1 + C_1/c_2)}, \frac {5}{2(2 + \sqrt{2}C_1/c_2)}, \frac {7}{2(3 + \sqrt{3}C_1/c_2)}, \frac{1}{2(1 + C_1(1+ C_3 )/c_2)} \right \}.
\end{align*}
\end{proof}

Let us return to a deterministic ordered rooted tree $\mathbb T$.
For a node $v \in \mathcal V(\mathbb T)$ define by $n(v)$ the size of the subtree rooted at $v$. 
For $M \geq 2$, let $\mathbb T_1$ be the binary subtree of the $2$-heavy tree of $\mathbb T$ containing all nodes with subtree sizes at least $M$.
Then, let $\mathcal V_2$ be set of nodes in the $2$-heavy tree of $\mathbb T$ with graph distance  1 from $\mathbb T_1$. By construction, $n(v) \leq M-1$ for $v \in \mathcal V_2$. Furthermore, let 
$\mathcal V_4$ be subset of nodes $v \in \mathcal V(\mathbb T)$ which are in a subtree rooted at a node in $\mathcal V_2$. (In particular, $\mathcal V_2 \subseteq \mathcal V_4$.) Next, let $\mathcal V_3 = \mathcal V(\mathbb T) \backslash (\mathcal V(\mathbb T_1) \cup \mathcal V_4)$ such that $|\mathbb T_1| + |\mathcal V_3| + |\mathcal V_4| = |\mathbb T|$. See Figure \ref{fig:bin} below for an illustration.
We denote the heavy binary tree in $\mathbb T$ by $\mathbb B$.
Note that, by construction, $|\mathcal V_4| \leq (M-1) |\mathcal V_2|$ and $|\mathbb B| \geq |\mathbb T_1| + |\mathcal V_2|$. Thus, 
\begin{align*} (M-1) |\mathbb B| \geq |\mathbb T_1| + (M-1) |\mathcal V_2| = |\mathbb T| -|\mathcal V_3| - |\mathcal V_4| + (M-1) |\mathcal V_2| & \geq |\mathbb T| -|\mathcal V_3|. \end{align*}
We arrive at the useful inequality, 
\begin{align*} |\mathbb B| \geq \frac{|\mathbb T| - |\mathcal V_3|}{ M-1}. \end{align*}
Let $\mathcal V_3(\tau_n)$ be $\mathcal V_3$ in the tree $\tau_n$. Then, with $0 < c < 1$, $$\Prob{|\mathcal V_3(\tau_n)| \geq c n} \to 0 \quad \text{implies}  \quad \Prob{B_n \leq n (1-c) / (M-1)} \to 0.$$

       \begin{figure}[h]  \centering
\includegraphics[width=.7\columnwidth]{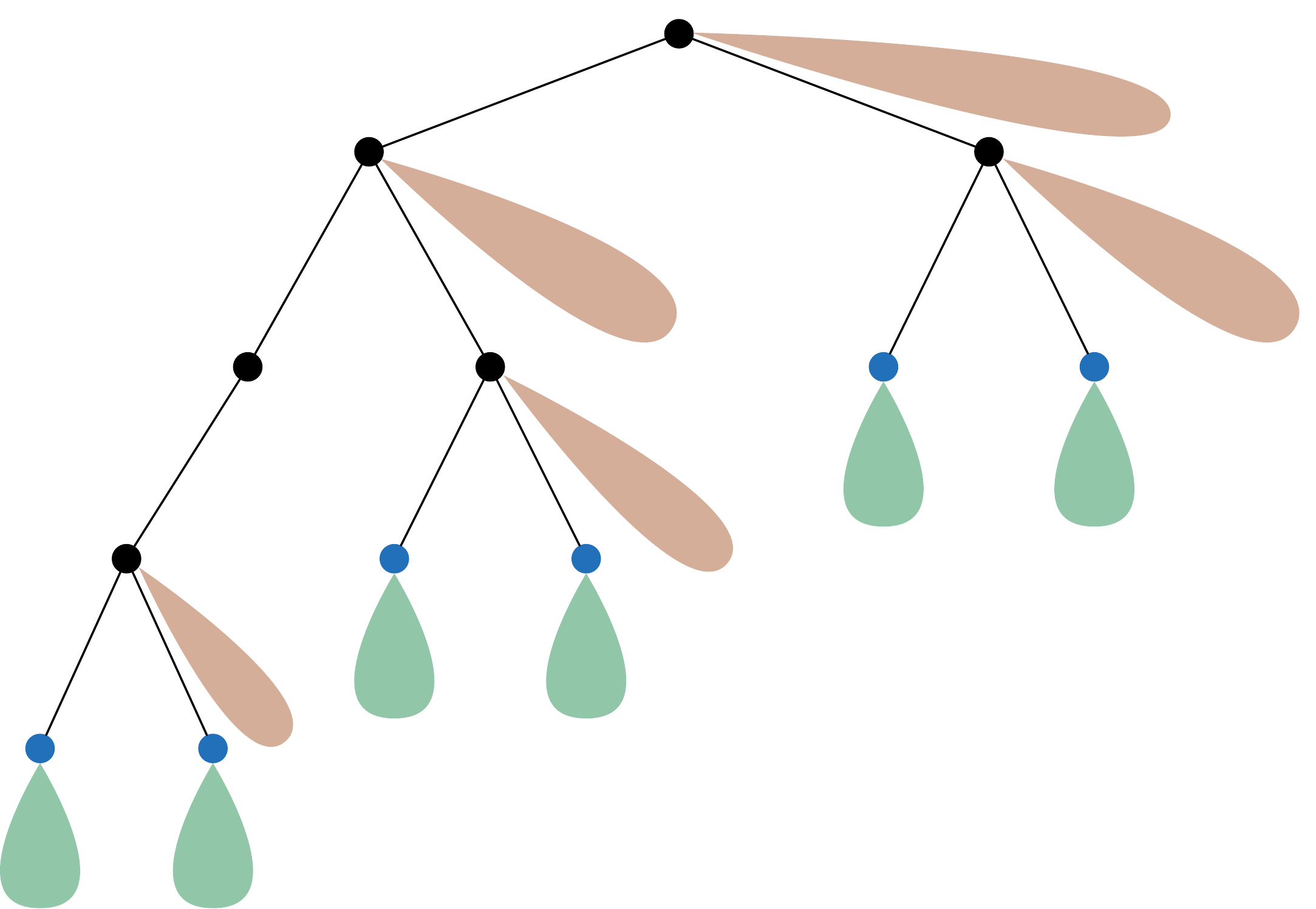} 
\caption{Instance of the construction underlying the proof Theorem \ref{thm:low2ary}. Black-filled nodes form $\mathbb T_1$, non-filled nodes constitute $\mathcal V_2$, dashed subtrees indicate
$\mathcal V_3$, and $\mathcal V_4$ is represented by the solid subtrees merged with $\mathcal V_2$.} 
 \label{fig:bin} \end{figure}

\begin{thm} \label{thm:low2ary}
Let $\E{\xi^5} < \infty$. Then, for some $0 < \gamma < 1$, we have $\Prob{|\mathcal V_3(\tau_n)| \geq \gamma n} \to 0$. Also, there exists a constant $\upsilon > 0$ such that $$\lim_{n \to \infty} \Prob{B_n \geq \upsilon n} = 1.$$
\end{thm}
\begin{proof} If $p_0 + p_1 + p_2 = 1$, the theorem is obviously correct. Thus, we assume $p_3 + p_4 + \ldots > 0$. 
For $1 \leq k \leq n$, $n \in I$, define $W_k = \sum_{v \in [n], N(v) = k} N_{3+}(v)$. By construction, $|\mathcal V_3(\tau_n)| \leq \sum_{k=M}^n W_k$. Let $\omega_n = o(n^{1/4}) $ be an integer sequence tending to infinity and, for $n \geq 0$, denote by
$\mathcal N_{3+}(n)$ a generic random variable with the distribution of $N_{3+}$ in $\tau_n$. (We abbreviate $\mathcal N_{3+}(n) = 0$ if $n \notin I$.)

Then, for $n$ sufficiently large, using Corollary \ref{logbound}, for some $C_1 > 0$, 
\begin{align*}
\Prob{ \sum_{k=\omega_n}^n W_k  \geq  \gamma n/2 } \leq \frac{2}{n \gamma} \sum_{k = \omega_n}^n \E{W_k} & = \frac{2}{n \gamma} \sum_{k = \omega_n}^n \E{Z_k} \E{\mathcal N_{3+}(k)} \\
& \leq \frac{2 C_1}{n \gamma} \sum_{k = \omega_n}^n \E{Z_k} \log k = O \left( \frac{\log \omega_n}{\sqrt{\omega_n}}\right).
\end{align*} 
Let $Y_1(k), Y_2(k), \ldots$ be independent copies of $\mathcal N_{3+}(k)$ and $W_k^* = \sum_{j =1}^{2 \E{Z_k}} Y_{j}(k)$. Then
\begin{align} 
\Prob{ \sum_{k=M}^{\omega_n} W_k  \geq \gamma n/2 } & \leq \mathbf{1} _{\left [ \gamma n / 8, \infty \right)}\left(\sum_{k = M}^{\omega_n} \E{Z_k} \E{\mathcal N_{3+}(k)} \right) \nonumber \\ 
&+ \omega_n \sup_{M \leq k \leq \omega_n} \Prob{W_k \geq 2 \E{W_k^*}}, \label{e1}
\end{align}
and,
\begin{align*}
\Prob{W_k \geq 2 \E{W_k^*}} & \leq \Prob{W_k \geq 2 \E{W_k^*}, Z_k \leq 2 \E{Z_k}} +  \Prob{Z_k \geq 2 \E{Z_k}}\\
&  \leq \Prob{W_k^* \geq 2 \E{W_k^*}} +  \Prob{Z_k \geq 2 \E{Z_k}}. 
\end{align*}
We use Chebyshev's inequality to bound both summands in the last expression. By \eqref{variance}, for some constant $C_2 > 0$ and all $M \leq k \leq \omega_n$, $k \in I_n$,
\begin{align*}
\Prob{Z_k \geq 2 \E{Z_k}} \leq \frac{\text{Var}(Z_k)}{\E{Z_k}^2} \leq C_2 \frac{\omega_n^3}{n}.
\end{align*}
Similarly, for some $C_3 > 0$,
\begin{align*}
 \Prob{W_k^* \geq 2 \E{W_k^*}} \leq  \frac{\text{Var}(W_k^*)}{\E{W_k^*}^2} = \frac{\E{\mathcal N_{3+}(k)^2}}{2 \E{Z_k} \E{\mathcal N_{3+}(k)}^2} \leq  C_3 \frac{\omega_n^2}{n}.
\end{align*}
Here, we have used the fact that $\liminf_{n \to \infty} \E{\mathcal N_{3+}} > 0$. Hence, the second summand in \eqref{e1} converges to zero as $n \to \infty$. By Corollaries \ref{logbound} and \ref{subtreebounds}, there exists a constant $C_4>0$ (depending on the offspring distribution but not on $M$) such that 
$\sum_{k = M}^{\omega_n} \E{Z_k} \E{\mathcal N_{3+}(k)} \leq C_4 n \log M / \sqrt{M}$. Choosing $M$ large enough such that $c_5 := 8 C_4 \log M / \sqrt{M}< 1$, the first assertion follows for any $\gamma \in (c_5, 1)$. Having picked $M$ and $\gamma$, we then have $$\lim_{n \to \infty} \Prob{B_n \leq n \frac{1-\gamma}{M-1}} = 0,$$
which proves the theorem.  
\end{proof}
We showed that with high probability, the heavy binary tree is larger than a positive constant times $n$. It implies that, under the equiprobable model of a random Apollonian tree, the longest simple path is with high probability $\Omega(n)$. This shows \eqref{conjfrieze}.

\section{Distances} \label{sec:distances}
The aim of this section is to give precise statements and proofs of the results claimed in Section \ref{sec:mainresults} on maximal distances. In particular, we will show that, under appropriate conditions, the maximal distance from the $2$-heavy tree is $\Theta(n^{1/3})$, a result that cannot possibly be deduced from the standard continuum random tree result for conditional Galton-Watson trees \cite{al1, al2, al3, legallrandomtrees}.

Let $\mathbb T$ be a ordered rooted tree. For $\mathcal V' \subseteq \mathcal V(\mathcal T)$, define 
$$\max  \mathcal V' := \max_{v \in \mathcal V(\mathbb T)} \text{dist}(v, \mathcal V'),$$
where $\text{dist}(\cdot, \cdot)$ denotes the graph distance on $\mathbb T$.
The main purpose of this section is to show the following result. Here, we write $A_k$ for the set of full $k$-ary subtrees containing the root.

\begin{thm} \label{thm:maindist}
 (i) Let $k \geq 3$. If $\E{\xi^{k+1}} < \infty$, then, for any $\varepsilon > 0$, there exists $C_1 > 0$ such that, for all $n \in I$, 
\begin{align} \label{bound:lowM}  \Prob{\max_{v \in [n]} N_k(v) \leq C_1n^{2/k}} \geq 1- \varepsilon. \end{align} 
If $\E{\xi^{k+1}} < \infty$ and $\sum_{\ell \geq k} p_\ell > 0$, then, for any $\varepsilon > 0$, there exists $c_1 > 0$ such that, for all $n \in I$, 
\begin{align} \label{bound:lowM2}  \Prob{\max_{v \in [n]} N_k(v) \geq c_1n^{2/k}} \geq 1- \varepsilon. \end{align} 
One can replace $N_k(v)$ by $N_{k+}(v)$ in \eqref{bound:lowM} upon possibly increasing $C_1$ if $\E{\xi^{(3k+1)/2}} < \infty$.

(ii)  Let $k \geq 2$. If  $\E{\xi^{k+3}} < \infty$, then, for any $\varepsilon > 0$,  there exists a constant $C_2 > 0$, such that, for all $n \in I$,
\begin{align}  \label{bound:lowH} 
\Prob{\max \mathcal V(\{1, \ldots, k\}^*) \leq C_2 n^{1/(k+1)}} \geq 1- \varepsilon,
\end{align}
If $\E{\xi^{k+2}} < \infty$ and $\sum_{\ell \geq k+1} p_\ell > 0$, then, for any $\varepsilon > 0$, there exists $c_2 > 0$ such that, for all $n \in I$, 
\begin{align}  \label{bound:lowsup} 
\Prob{\inf_{T \in A_{k}} \max \mathcal V(T) \geq c_2n^{1/(k+1)}} \geq 1- \varepsilon.
\end{align}

In other words, under the assumptions imposed, 
the sequences $n^{-1/(k+1)} \max \mathcal V(\{1, \ldots, k\}^*),$    $n^{-1/(k+1)}$   $ \inf_{T \in A_{k}}$ $\max \mathcal V(T)$ and $n^{-2/k} \max_{v \in [n]} N_k(v)$ as well as their reciprocals are tight. 
\end{thm}

\medskip Let us briefly discuss the result. First of all, the lower bounds \eqref{bound:lowM2} and \eqref{bound:lowsup} are much harder to obtain than the upper bounds \eqref{bound:lowM} and \eqref{bound:lowH}, where  \eqref{bound:lowsup} follows very easily from \eqref{bound:lowM2} from known tail bounds on the height of $\tau_n$ (see inequality \eqref{boundh} in the introduction). Second, on the one hand,
\eqref{bound:lowsup} says that \emph{every} $k$-ary subtree of $\tau_n$ (and not only the $k$-heavy tree $\mathcal V(\{1, \ldots, k\}^*)$) leaves out nodes of distance order $n^{1/k}$ away. On the other hand, \eqref{bound:lowH} shows that the 
 $k$-heavy tree   exhausts $\tau_n$ in an asymptotically optimal way. Third, in light of Theorem \ref{thm:tailbounds2}, the moment conditions imposed in 
\eqref{bound:lowM2} and \eqref{bound:lowsup} are somewhat unexpected. Indeed, we believe that these results are valid under a finite variance assumption on the offspring distribution. However, since our proof relies on the second moment method and involves suitable bounds on variances, we cannot remove these conditions.

\medskip The next proposition gives estimates on the sizes of sets which are, in a certain sense, \emph{close} to the heavy path. In order to make this more precise, we need to introduce some more notation.
If $\{k+\}$ denotes $\{k,k+1, \ldots \}$, then $\mathcal V(\{k + \}^*)$ is the vertex set of the subtree of $\mathbb T$ that avoids any node of index strictly smaller than $k$ (the root excepted). We are also interested in $$\mathcal V(\underbrace{\mystrut{1.0ex} 1^*21^*2\ldots 1^*2}_{k \ \text{pairs} \ {1^*2}} 1^*) = \mathcal V\left((1^*2)^k 1^*\right),$$
where $(s)^k$ denotes $k$ repetitions of a sequence $s$. We use the simplified notation $|A|$ to denote the number of nodes $v \in \mathcal V(\mathbb T)$ with $\kappa(v) \in A$.  Note that we have the following disjoint decompositions, 
\begin{align*}
 \mathcal V(\{1,2\}^*) =  \bigcup_{k = 0}^\infty \mathcal V\left((1^*2)^k 1^*\right), \quad
\mathcal V(\mathbb T)  = \mathcal V\left(\{1,2\}^*\right)  \cup \bigcup_{k=0}^\infty \mathcal V\left((1^*2)^k 1^*({3+}) \N^*\right) .
\end{align*}
From now on, we apply these definitions to the tree $\tau_n$. From Theorem \ref{thm:meanbinary}, we know that, \sloppy $\E{|\mathcal V(\{1,2\}^*)|} = \Omega(n)$. Obviously, $\E{ | [n]  \backslash \mathcal V(\{1,2\}^*)|} = \Omega(n)$ if $p_0 +p_1 + p_2  < 1$.  In the next section, we will determine the first order behaviour of $\mathcal V(1^*)$: from Theorem \ref{conv:hp} (or, also from \eqref{boundh}), it follows that $\E{|\mathcal V(1^*)|} = \Theta(\sqrt{n})$.
\begin{prop} \label{thm:general}
As $n \to \infty$, $n \in \N h +1$, 
\begin{itemize}
\item [(i)] if $\E{\xi^{7/2}} < \infty$, then $\E{| \mathcal V\left((1^*2)^k 1^*\right)|} = \Theta(\sqrt{n} \log^k n), k \geq 1$, 
\item  [(ii)] if $\E{\xi^{5}} < \infty$, then $\E{| \mathcal V\left((1^*2)^k 1^*({3+}) \N^*\right)|} = O(\sqrt{n} \log^{k+1} n), k \geq 0$,
\item [(iii)] if $\E{\xi^{13/2}} < \infty$, then $\E{| \mathcal V\left((1^*2)^k 1^*({4+}) \N^*\right)|} = O(\sqrt{n})$, $k \geq 0$, and
\item  [(iv)] if $\E{\xi^{3}} < \infty$, then $\E{ |\mathcal V(\{2+\}^*)|} = O(1).$
\end{itemize}
If $\sum_{\ell \geq 3} p_\ell > 0$, then big-$O$ in (ii) can be replaced by $\Theta$. Similarly, big-$O$ in (iii) becomes $\Theta$ if $\sum_{\ell \geq 4} p_\ell > 0$.
\end{prop}
Together with Lemma \ref{lem:maxsub} below, the proposition shows the following result:  for any fixed $M \in \N$, with high probability, the index sequence of every node $v \in [n]$ with extraordinarily large $k$-th subtree size contains at least $M$ entries different from $1$.

The proof of the proposition is worked out in Appendix B. The rest of this section is devoted to the proof of Theorem \ref{thm:maindist}.

\subsection{Upper bounds} \label{sec:up}
 For a node $ v \in [n]$ with index sequence $\kappa(v)$, we define
$H^*(\kappa(v)) := H(v)$ and $N^*(\kappa(v)) := N(v)$. Here, we recall that $H(v)$ denotes the height of the subtree
rooted at $v$ in $\tau_n$. (We also abbreviate $H^*(y) = N^*(y) = 0$ for $y \in \N^*$ when $\kappa(v) \neq y$ for all nodes $v \in [n]$.)

\begin{lem} \label{lem:conheight}
Let $k \geq 2$ and $\E{\xi^{k + 2}} < \infty$. Then, there exists a constant $C > 0$ such that, for all $n \in I$ and $t \geq 1$, 
$$\Prob{ \sup_{\ell \geq k} H^*(\ell) \geq t} \leq C t^{1-k}.$$
\end{lem}
\begin{proof} Let $\{\mathcal H_{i}(n) : n \in I, i \geq 1 \}$ be a family of independent random variables where each $\mathcal H_i(n)$ is distributed like the height of $\tau_n$. Furthermore, assume that the family is independent of $\tau_n$. Using \eqref{boundh}, we have
\begin{align*}
\Prob{ \sup_{\ell \geq k} H^*(\ell) \geq t} & = \E{ \Prob{\sup_{k \leq \ell \leq \xi_\epsilon} H^*(\ell) \geq t \Bigg | \xi_\epsilon, N_1, N_2, \ldots}} \\
& \leq \E{\xi_\epsilon \sup_{k \leq \ell \leq \xi_\epsilon} \Prob{\mathcal H_{\ell}(N_\ell) \geq t \bigg | \xi_\epsilon, N_1, N_2, \ldots}} \\
& \leq \E{\xi_\epsilon \exp(-\delta t^2 / N_k)} \\
& = \sum_{\ell = k}^\infty \int_0^\ell  \Prob{N_k \geq - \delta t^2 / \log (s/\ell), \xi_\epsilon = \ell} ds.
\end{align*}
By inequality \eqref{tailrem} in the remark following Theorem \ref{thm:tailbounds}, there exists $C_1 >0$ such that the right-hand side of the last display is bounded from above by  
\begin{align*}
 \sum_{\ell = k}^\infty C_1 p_\ell \ell^{k+2}  \int_0^1   \left(-\frac{\delta t^2}{\log s}\right)^{(1-k)/2} ds \leq C_1 \Gamma((k+1)/2) \E{\xi^{k+2}} \delta^{(1-k)/2} t^{1-k}.
\end{align*}
Here, $\Gamma(x) = \int_0^\infty e^{-t} t^{x-1} dt$ denotes the Gamma function. This concludes the proof.
\end{proof}

\begin{prop} \label{upperboundM} Let $k \geq 2$ and $\E{\xi^{k+1}} < \infty$. Then, there exists a constant $C > 0$ such that, for all $n \in I$ and $t \geq 1$,
\begin{align*}
\Prob{\max_{v \in [n]} N_k(v) \geq t} \leq C \frac{n}{t^{k/2}}. \end{align*}
The bound also holds for 
$N_{k+}(v)$ if $\E{\xi^{(3k+1)/2}} < \infty$ upon possibly increasing $C$.

\end{prop}

\begin{proof}
The left hand side is zero for $t \geq \lceil n/2 \rceil$. Thus, we assume $t \leq \lceil n/2 \rceil - 1$.
Note that, for all nodes $v \in [n]$ with $N(v) \geq \lceil n/2 \rceil$, we must have $v \in \mathcal V(1^*)$. Hence, there are at most $| \mathcal V(1^*)| $ many of them in the tree.
Thus, by Theorem \ref{thm:tailbounds}, 
\begin{align*}
\Prob{\max_{v \in [n]} N_k(v) \geq t} & \leq \E{ | \{v \in [n] :  N_k(v) \geq t \}|} \\
& = \sum_{i = 1}^n \Prob{N_k(i) \geq t} \\
&  \leq \beta_k t^{(1-k)/2} \E{| \{v \in [n] :  N(v) \geq t\}|} \\
& \leq  \beta_k t^{(1-k)/2}  \left( \sum_{\ell = t}^{\lceil n/2 \rceil - 1}  \E{Z_\ell} + \E{| \mathcal V(1^*)|  } \right) \\
&  \leq C_1 \left(\frac{n}{t^{k/2}}  + \frac{\sqrt{n}}{t^{(k-1)/2}} \right),
\end{align*}
where $C_1$ can be chosen independently of $t$ and $n$ by Corollary \ref{bounds:mean} and the fact that $\E{| \mathcal V(1^*)| } = O(\sqrt{n})$. The same argument applies to $N_{k+}(v)$.
\end{proof}

Proposition \ref{upperboundM} is sufficient to deduce the upper bound in \eqref{bound:lowM}. In order to transfer the result to distances, we need a tighter bound when restricting to nodes on the heavy path.

\begin{lem} \label{lem:maxsub}
Let $k \geq 2$ and $\E{\xi^{k+1}} < \infty$. Then, for any deterministic (possibly infinite) set $\mathcal A \subseteq \N^*$ and $t \geq 1$, 
\begin{align*}
\Prob{\max_{ v \in \mathcal V(\mathcal A)} N_k(v) \geq t} \leq \beta_k t^{(1-k)/2} \E{| \mathcal V(\mathcal A)|},
\end{align*}
with $\beta_k$ as in \eqref{firsttail}. If $\E{\xi^{(3k+1)/2}} < \infty$, then the bound also holds with $k$ replaced by $k+$ (and $\beta_k$ by $\beta_{k+}$).
Furthermore, if $\E{\xi^{k+2}} < \infty$, then there exists a constant $C>0$ such that, 
 \begin{align*}
\Prob{\max_{ v \in \mathcal V(\mathcal A), \ell \geq k} H^*(\kappa(v)\ell) \geq t} \leq C t^{1-k} \E{| \mathcal V(\mathcal A)|}.
\end{align*}
 \end{lem}

The corollary follows immediately from  Theorem \ref{conv:hp}.
\begin{cor} \label{cor:maxsub}
Let $k \geq 2$ and $\E{\xi^{k+1}} < \infty$. Then, there exists a constant $C_1 > 0$ such that
\begin{align*}
\Prob{\max_{ v \in \mathcal V(1^*)} N_k(v) \geq t} \leq C_1 \frac{\sqrt n }{t^{(k-1)/2}}.
\end{align*}
If $\E{\xi^{(3k+1)/2}} < \infty$, then the same results hold with $N_k(v)$ replaced by $N_{k+}(v)$ upon possibly increasing $C_1$. Finally, if $\E{\xi^{k+2}} < \infty$, then there exists $C_2 > 0$, such that
\begin{align*}
\Prob{\max_{ v \in \mathcal V(1^*), \ell \geq k} H^*(\kappa(v)\ell) \geq t} \leq C_2 \frac{\sqrt n }{t^{k-1}}.
\end{align*}

\end{cor}

\begin{proof}  [Proof of Lemma \ref{lem:maxsub}]
For $\ell \geq 0$, let  $\mathcal A_{\ell,n}$ be the subset of $\mathcal A$ of vectors of length $\ell$ where each entry is bounded from above by $n$. 
We have
\begin{align} \label{eq:start1}
\Prob{\max_{ v \in \mathcal V(\mathcal A)} N_k(v) \geq t} \leq \sum_{\ell = 0}^n \Prob{\max_{v \in  \mathcal V(\mathcal A_{\ell,n})} N_k(v) \geq t }. 
\end{align}
We denote the elements of $ \mathcal A_{\ell,n}$ by $y_1, \ldots, y_K$, $K  = K(\ell) \leq n^\ell$. Let $\{\mathcal N_k^{(i)}(j): i \geq 1, j \in I\}$ be a family of independent random variables where each $\mathcal N_k^{(i)}(j)$ is distributed like $N_k$ in the tree $\tau_j$. Then, 
using \eqref{firsttail}, 
\begin{align*}
& \Prob{\max_{v \in  \mathcal V(\mathcal A_{\ell,n})} N_k(v) \geq t }  \\
&= \sum_{0 \leq n_1, \ldots, n_K \leq n} \Prob{\max_{v \in \mathcal V(\mathcal A_{\ell,n})} N_k(v) \geq t \bigg | \bigcap_{j=1}^K \{N^*(y_j) = n_j\} } \Prob{\bigcap_{j=1}^K \{N^*(y_j) = n_j\}} \\
& = \sum_{0 \leq n_1, \ldots, n_K \leq n} \Prob{\max_{1 \leq j \leq K} \mathcal N_k^{(j)}(n_j) \geq t} \Prob{\bigcap_{j=1}^K \{N^*(y_j) = n_j\}} \\
& \leq \sum_{0 \leq n_1, \ldots, n_K \leq n} | \{1 \leq j \leq K: n_j \geq t\}| \sup_{1 \leq i \leq K} \Prob{\mathcal N^{(1)}_k(n_i) \geq t} \Prob{\bigcap_{j=1}^K \{N^*(y_j) = n_j\}} \\
& \leq \beta_k t^{(1-k)/2} \E{|\{v \in \mathcal V(\mathcal A_{\ell,n}) : N(v) \geq t\}|}.
\end{align*}
Plugging the bound into \eqref{eq:start1} gives
\begin{align*}
\Prob{\max_{ v \in \mathcal V(\mathcal A)} N_k(v) \geq t} \leq \beta_k t^{(1-k)/2} \E{|\{v \in \mathcal V(\mathcal A) : N(v) \geq t\}|} \leq \beta_k t^{(1-k)/2} \E{|\mathcal V(\mathcal A)|}.
\end{align*}
The same proof works for $N_{k+}(v)$. Similarly, one obtains the result for the heights upon replacing $N_k(v)$ by $\max_{\ell \geq k} H^*(\kappa(v)\ell)$ and using Lemma \ref{lem:conheight}.
\end{proof}

\begin{prop} \label{upperboundH} Let $k \geq 2$ and $\E{\xi^{k+2}} < \infty$. Then, there exists a constant $C>0$ such that, for $t \geq 1$, $n \in I$,  
\begin{align*} 
\Prob{\max_{v \in [n], \ell \geq k} H^*(\kappa(v)\ell) \geq t} \leq C \left( \frac{n}{t^{k}} + \frac{\sqrt {n}}{t^{k-1}}\right). \end{align*}
\end{prop}

\begin{proof} We may assume $t \geq n_0$ with $n_0$ as in Lemma \ref{bounds:mean}. Taking the maximum only over nodes $v \in \mathcal V(1^*)$, the claim follows from Corollary \ref{cor:maxsub}. For $k \geq 1, n \in I$, let $(\mathcal H(n), \mathcal N_k(n), \bar \xi(n))$ be distributed like $(H, N_k, \xi_\epsilon)$ in $\tau_n$. Using \eqref{boundh}, we have
\begin{align*} & \Prob{\max_{v \in [n] \backslash \mathcal V(1^*), \ell \geq k} H^*(\kappa(v)\ell) \geq t} \\ & \leq \E{ | \{v \in [n]  \backslash \mathcal V(1^*) :  H^*(\kappa(v)\ell) \geq t \ \text{for some} \ \ell \geq k\}|} \\
& = \sum_{i = 1}^n \Prob{H^*(\kappa(i)\ell ) \geq t \ \text{for some} \ \ell \geq k, i \notin \mathcal V(1^*)} \\
& \leq  \sum_{i = 1}^{n} \sum_{\ell = k}^n \sum_{m = t}^{\lceil n / 2 \rceil -1} \sum_{j=t}^m\Prob{\mathcal H(j) \geq t} \Prob{\mathcal N_\ell(m) = j, \ell \leq \bar \xi(m)} \Prob{N(i) = m} \\
& \leq \sum_{i = 1}^n \sum_{\ell = k}^n \sum_{m = t}^{\lceil n / 2 \rceil -1}   \E{\exp\left( -\frac{\delta t^2}{\mathcal N_\ell(m)}\right) \mathbf{1}_{ t \leq \mathcal N_\ell(m), \ell \leq \bar \xi(m)}} \Prob{N(i) = m}. 
\end{align*}
The expectation in the last display is bounded by
\begin{align*}
& \int_{0}^{e^{-\delta t^2/m}}  \Prob{\exp\left( -\frac{\delta t^2}{\mathcal N_\ell(m)}\right) \geq x, \ell \leq \bar \xi(m)} dx \\ &=   \int_{0}^{e^{-\delta t^2/m}} \Prob{\mathcal N_k(m) \geq \frac{\delta t^2}{\log 1/x}, \ell \leq \bar \xi(m)} dx. 
\end{align*}
By Theorem \ref{thm:tailbounds},  there exists $C_1 > 0$ such that 
\begin{align*}
 \sum_{\ell = k}^n  \int_{0}^{e^{-\delta t^2/m}} \Prob{\mathcal N_k(m) \geq \frac{\delta t^2}{\log 1/x}, \ell \leq \bar \xi(m)} dx &  \leq  \int_{0}^{e^{-\delta t^2/m}}  \E{\bar \xi(m) \mathbf{1}_{\mathcal N_k(m) \geq \frac{\delta t^2}{\log 1/x}}} dx \\
& \leq C_1 \delta^{(1-k)/2} t^{1-k} \int_{0}^{e^{-\frac{\delta t^2}{m}}} \left(\log \frac 1 x \right)^{\frac {k-1}{2}} dx \\
& \leq C_2 m^{(1-k)/2} e^{-\delta t^2 / m }.
\end{align*}
Here, $C_2 > 1$ denotes some constant which is independent of $m, t$ and $n$. 
Summarizing and using Corollary \ref{bounds:mean}, we obtain
\begin{align*}
\Prob{\max_{v \in [n] \backslash \mathcal V(1^*), \ell \geq k} H^*(\kappa(v)\ell) \geq t} & \leq C_2 \sum_{m = t}^{\lceil n / 2 \rceil -1} \E{Z_m} ( m^{(1-k)/2} e^{-\delta t^2 / m } + e^{-\delta t}) \\
& \leq 2 \sqrt{2} C_2 \alpha n\sum_{m = t}^{\lceil n / 2 \rceil -1}  m^{-1-k/2} e^{-\delta t^2/ m }  \leq C_3 n t^{-k},
\end{align*}
for some $C_3 > 0$. 
This concludes the proof.
\end{proof}

\subsection{Lower bounds} \label{sec:down}
Our lower bounds rely on a variant of the second moment method which requires sufficiently tight upper bounds on variances (or second moments). To this end, we use Lemma 6.1 in Janson \cite{janson2013} and introduce the notation used in this work. 
Denote by $\mathfrak T$ the set of all ordered rooted trees. For a function $f : \mathfrak T \to \mathbb R$, let $F$ be defined by
$$F(\mathbb T) := F(f, \mathbb T): = \sum_{v \in \mathcal V(\mathbb T)} f(\mathbb T_v), \quad \mathbb T \in \mathfrak T.$$
Here $\mathbb T_v$ denotes the fringe tree in $\mathbb T$ rooted at $v$. 
For $k \geq 1$, we abbreviate $f_k(\mathbb T) := f(\mathbb T) \mathbf{1}_{|\mathbb T| = k}$. Note that $F(f_k, \tau_n) = Z_k$ for $f = \mathbf 1$, where $\mathbf 1$ denotes the function $f$ mapping every tree to $1$.
 Then, for $1 \leq m \leq k \leq n/2$, 
\begin{align*} 
 \text{Cov}  (F(f_k, \tau_n),  F(f_m, \tau_n)) = I_1(f,k,m) +  I_2(f,k,m) +I_3(f,k,m), \end{align*}
 where
 \begin{align*} 
I_1(f,k,m) = & \frac{n\Prob{S_{n-k} = 0} \Prob{S_k = -1}}{k\Prob{S_n = -1}}   \E{f_k(\tau_k) F(f_m,\tau_k)}, \\
I_2(f,k,m) = &  \frac{n(n-k-m+1)}{mk} \Prob{S_k = -1} \Prob{S_m = -1}   \E{f_k(\tau_k)} \E{f_m(\tau_m)} \cdot \\
& \left( \frac{\Prob{S_{n-k-m} = 1}}{ \Prob{S_n = -1}}  - \frac{\Prob{S_{n-k} = 0}}{ \Prob{S_n = -1}} \frac{\Prob{S_{n-m} = 0}}{ \Prob{S_n = -1}}\right), \\
\end{align*}
and
\begin{align*}
I_3(f,k,m) =  & - \frac{n(k+m-1)}{mk} \frac{\Prob{S_{n-k} = 0}}{ \Prob{S_n = -1}} \frac{ \Prob{S_{n-m} = 0}}{\Prob{S_n = -1}} \cdot \\ 
&  \Prob{S_k = -1} \Prob{S_m = -1}  \E{f_k(\tau_k)} \E{f_m(\tau_m)}. 
\end{align*}
Note that, by the crucial Lemma 6.2 in \cite{janson2013}, cancellation effects in $I_2(f,k,m)$ cause this term to be of the order $n$ (for $m,k$ fixed), rather than $n^2$. Below, we only need upper bounds on the variance which allows us to neglect $I_3(f,k,m)$. For $i = 1, 2$, we set $I_i(k,m) = I_i(\mathbf 1, k,m)$.

For $1 \leq t \leq n$, $t \in \N$, we define $$\mathcal Z_t = | \{v \in [n] :  t \leq N(v) \leq 2t\}| = \sum_{\ell = t}^{2t} Z_\ell.$$ From Corollary \ref{bounds:mean}, we know that there exist a constant $K_1 > 0$ only depending on the offspring distribution such that, for all $1 \leq t \leq n/4$, we have 
\begin{align} \label{inq:K1} K_1^{-1} \frac{n}{\sqrt{t}} \leq \E{\mathcal Z_t} \leq K_1 \frac{n}{\sqrt{t}}. \end{align}

\begin{prop} \label{propeasy} There exists a constant $C > 0$, such that, for all $1 \leq t \leq (n-1)/4, t \in \N$ and $n \in I$, we have 
$$\emph{Var}(\mathcal Z_t)  \leq C n.$$
In particular, for any $t = t(n)$ with $t = o(n)$, we have, as $n \to \infty$, in probability, 
$$\frac{\mathcal Z_t}{\E{\mathcal Z_t}} \to 1.$$

\end{prop}
\begin{proof}
We use the notation introduced above with the function $f = \mathbf 1$. 
Obviously, $$\text{Var}(\mathcal Z_t) = \sum_{k,m = t}^{2t} \text{Cov}(Z_k, Z_m) \leq 2 \sum_{m= t}^{2t}\sum_{k=m}^{2t} \text{Cov}(Z_k, Z_m). $$ 
In the following, $C_i, i \geq 1,$ denote constants independent of $k, m, t$ and $n$, whose precise values are of no relevance. For $m \leq k$,  by the local limit theorem \eqref{LLTgen}, we have
 $I_1(k,m) \leq C_1 n m^{-3/2} (\max(1,k-m))^{-1/2}$. Thus, $$\sum_{m=t}^{2t} \sum_{k=m}^{2t} I_1(k,m) \leq C_2 n \sum_{m=t}^{2t} m^{-3/2} \sqrt{2t - m} \leq C_3 n.$$ 
By Lemma 6.2 in \cite{janson2013}, for $t \leq m \leq k \leq 2t$,
\begin{align} \label{boundI3}
I_2(k,m) &\leq C_4 n^2 ((km)^{-3/2} \left(\frac 1 n + \frac{k+m}{n^{3/2}} + \frac{km}{n^2}\right) \leq C_5 t^{-2} (nt^{-1} + \sqrt{n} + t).
\end{align}
Hence, $\sum_{m=t}^{2t} \sum_{k=m}^{2t} I_2(k,m) \leq C_6 (nt^{-1} + \sqrt{n} + t) \leq C_7 n$. This finishes the proof. \end{proof}

For $\ell \geq 2$ and $t > 0$, let $g_\ell(\mathbb T) = \mathbf{1}_{n_\ell(\mathbb T) \geq t}$ where $n_\ell(\mathbb T)$ denotes the size of the $\ell$-th biggest subtree of the root of $\mathbb T$.  (We suppress $t$ in the notation.) For $i \geq 1$, define $F^*_i(\cdot) = F((g_\ell)_i, \cdot)$. Further, for $t > 0$, let $t' = \lfloor (\ell+1) t \rfloor$. 
Finally, let $V_t = \sum_{i= t'}^{2t'} F^*_i(\tau_n)$. 
Then, \begin{align} \label{meanZ} \E{V_t} = \sum_{i = t' }^{ 2t'} \Prob{\mathcal N_\ell(i) \geq t} \E{Z_i}, \end{align}
where, as before,  we write $\mathcal N_\ell(i)$ for a random variable distributed like $N_\ell$ in $\tau_i$.

\begin{prop}

Let $\ell \geq 2$. 

\begin{itemize} \item [(i)] If $\E{\xi^{\ell + 1}} < \infty$, then, there exists a constant $C_1  > 0$, such that, for $ n \in I$ sufficiently large and $t \leq n/4$, $$ \E{V_t}  \leq C \frac{n}{t^{\ell/2}}.$$
\item [(ii)] If $\sum_{m \geq \ell} p_m > 0$, then, there exist constants $C_2, K_2 > 0$, such that, for $ n \in I$ sufficiently large, and $C_2 \leq t \leq n  / (4(\ell+1))$,
$$\E{V_t} \geq K_2^{-1} \frac{n}{t^{\ell/2}}.$$
 \item [(iii)]  If $\E{\xi^{\ell + 1}} < \infty$, then there exists a constant $K_3 > 0$ such that, for all $n \in I, 0 < t < (n-1)/4$, we have $$\emph{Var}(V_t) \leq (1 + K_3 t^{(3-\ell)/2}) \E{V_t} + K_3 \left(n t^{- \ell} + \sqrt{n} t^{1-\ell} + t^{2-\ell}\right).$$
\end{itemize}

\end{prop}
\begin{proof}
The bounds on the mean in (i) and (ii) immediately follow from \eqref{meanZ} and the bounds in \eqref{inq:K1} using the tail bounds in Theorems \ref{thm:tailbounds} and \ref{thm:tailbounds2}.
In (iii), we may assume $\sum_{m \geq \ell} p_m > 0$, since,  otherwise, $V_t = 0$ almost surely. We then have $$\text{Var}(V_t)  \leq \sum_{m=t'}^{2t'}  (I_1(g_\ell, m,m) + I_2(g_\ell, m,m)) + 2\sum_{m= t'}^{2t'}\sum_{k=m+1}^{2t'} (I_1(g_\ell, k,m) + I_2(g_\ell, k,m)),$$ where
\begin{align} \label{def:ilkm}
I_1(g_\ell, k,m) & = \frac{n\Prob{S_{n-k} = 0}}{k\Prob{S_n = -1}}  \Prob{S_k = -1} \E{ \mathbf{1}_{ N^*_\ell \geq t} |\{v \in \tau_k : N^*(v) = m, N^*_\ell(v) \geq t \}|}, 
\end{align}
and
\begin{align*}
 I_2(g_\ell, k,m) & =  
 I_2(k,m) \Prob{\mathcal N_\ell(k) \geq t} \Prob{\mathcal N_\ell(m) \geq t}.
\end{align*}
In \eqref{def:ilkm}, $^*$ is used on the right-hand side  to indicate that the quantities are considered in the tree $\tau_k$.  Combining the bounds in Theorem \ref{thm:tailbounds} and \eqref{boundI3}, there exists $C_1 > 0$ such that $$\sum_{m= t'}^{2t'} \sum_{k=m}^{2t'} I_2(g_\ell, k,m) \leq C_1 (n + \sqrt{n} t + t^2) t^{- \ell}.$$ Next, again using Theorem \ref{thm:tailbounds},
\begin{align*} & \sum_{m= t'}^{2t'} \sum_{k=m+1}^{2t'} I_1(g_\ell, k,m) \\ & = \E{ | \{(v,w) \in \tau_n^2:  t' \leq N(v), N(w) \leq 2t', N_\ell(v) \geq t, N_\ell(w) \geq t, w \in \tau(v), w \neq v \}| } \\
& \leq \beta_\ell t^{(1-\ell)/2}  \E{ | \{(v,w) \in \tau_n^2: t' \leq N(v) \leq 2t', N_\ell(v) \geq t, w \in \tau(v), w \neq v \}| } \\
& \leq 2\beta_\ell (\ell+1) t^{(3-\ell)/2} \E{ | \{v \in [n]: t' \leq N(v) \leq 2t', N_\ell(v) \geq t \}| } \\ 
& = 2 \beta_\ell  (\ell+1) t^{(3-\ell)/2} \E{V_t}.
\end{align*}
Finally, $\sum_{m=t'}^{2t'}  I_1(g_\ell, m,m) = \E{V_t}$. 
This concludes the proof.
\end{proof}

\begin{proof} [Proof of Theorem \ref{thm:maindist}] The upper bounds \eqref{bound:lowM} and \eqref{bound:lowH}  follow immediately from Propositions \ref{upperboundM}  and \ref{upperboundH}. 
For the lower bound in \eqref{bound:lowM2}, let $\ell \geq 3$, and note that, by Chebyshev's inequality, using the bounds in the previous proposition, for $t$ and $n$ sufficiently large with $t \leq n / (4(\ell+1))$, 
$$\Prob{V_t = 0} \leq \Prob{|V_t - \E{V_t}| \geq \E{V_t}} \leq \frac{\text{Var}(V_t)}{\E{V_t}^2} \leq \frac{1 + K_3}{\E{V_t}} + K_2^2 K_3 \left( \frac {1}{n} + \frac{t}{n^{3/2}} + \frac{t^2}{n^2}\right).$$
Now, \eqref{bound:lowM2} follows upon choosing $t = c n^{2/\ell}$ with $c > 0$ sufficiently small.
For the lower bound in \eqref{bound:lowH} note that, for $\varepsilon > 0$, there exists $n_3 > 0$ such that, for all $n \geq n_3$, we have $\Prob{H \geq \varepsilon \sqrt{n}} \geq 1- \varepsilon$. Hence, for $n_3 \leq m \leq n$, 
\begin{align*}
& \Prob{\max_{v \in [n]} \min_{1\leq i \leq \ell} H^*(\kappa(v) i ) \geq \varepsilon \sqrt{m}} \\ & \geq \Prob{\max_{v \in [n]} \min_{1\leq i \leq \ell} H^*(\kappa(v) i )  \geq \varepsilon \sqrt{m}, \max_{v \in [n]} N_\ell(v) \geq m} \\
& \geq  \sum_{j = 1}^n \sum_{m' =m}^n \Prob{ \min_{1\leq i \leq \ell}  H^*(\kappa(j)i) \geq \varepsilon \sqrt{m},  N_\ell(j) = m', N_\ell(j') < m \ \text{for all } 1 \leq j' < j} \\
& \geq (1-\varepsilon)^\ell \Prob{\max_{v \in [n]} N_\ell(v) \geq m}. 
\end{align*}
Hence, the lower bound in \eqref{bound:lowH} follows from the lower bound in \eqref{bound:lowM} upon choosing $m = c_1 n^{2/\ell}$ in the last display with $c_1 > 0$ sufficiently small. \end{proof}
\section{The heavy path} \label{sec:heavypath}

In this section, we study $\mathcal V(1^*)$. We set $L_n = |\mathcal V(1^*)| - 1$ as in the introduction. Recall from Section \ref{sec:twopic} (i), that the scaling limit of conditional Galton-Watson trees is the continuum random tree. More precisely, with the depth-first search process $(D_t)_{0 \leq t \leq 2n-2}$ defined in Section \ref{sec:encoding} and endowing the space of continuous functions with the supremum norm, we have, 
\begin{align} \label{conv:ex}
\left(\frac{D_{t(2n-2)}}{\sqrt n}\right)_{0 \leq t \leq 1} \stackrel{d}{\longrightarrow} \frac{2}{\sigma} \cdot \mathbf e,
\end{align}
where $\mathbf e$ is a standard Brownian excursion. This is Aldous's Theorem 2 \cite{al2}.
As already indicated in the introduction, the heavy path can be defined in the  continuum random tree making use of its definition based on Brownian excursion. Therefore, using \eqref{conv:ex}, convergence of $n^{-1/2} L_n$ boils down to an application of the continuous mapping theorem.
The technical steps in this context leading to the following theorem are intricate and of entirely different flavor than the arguments in the rest of the paper. Therefore, we defer the analysis to Appendix D. A representation of the limiting random variable, also stated in Appendix D, leading to the explicit formula for the moments relies on arguments from self-similar fragmentation processes, in particular, on the work of Bertoin \cite{bertoin_hom, bertoin_self} and Carmona, Petit and Yor \cite{carpetyor}.

\begin{thm} \label{conv:hp}
As $n \to \infty$, in distribution and with convergence of all moments, 
$$\frac{L_n}{\sqrt n} \to \frac 2 \sigma \cdot T_\infty,$$
where, for $k \geq 1$, 
$$\E{T_\infty^k} = \frac{k!}{\Phi(\frac 1 2) \cdots \Phi(\frac k 2)},$$
with $\Phi(q) = \frac{4}{\sqrt{\pi}}  \cdot {}_2F_1\left(-\frac 1 2, \frac 3 2 -q; \frac 1 2; \frac 1 2\right), q > 0$. Here, ${}_2F_1$ denotes the standard hypergeometric function. 
The distribution of $T_\infty$ is characterized by the stochastic fixed-point equation  \eqref{perp}.
\end{thm}

\medskip \textbf{Remark.} It turns out that, in distribution, $T_\infty = \int_0^\infty e^{-\frac 1 2 \xi(t)} dt$ for some 
non-negative subordinator $\xi(t), t \geq 0$. In Theorem \ref{thm:unif_conv_exc}, we also state 
functional limit theorems (after rescaling) for the quantities 
$$P_n(k) = \{N(v) : v \text{ has distance } k \text{ from the root  and }\kappa(v) =1\ldots 1 \}, \quad k \geq 1$$ 
and $Q_n(\ell) = \inf \{ k \geq 0: P_n(k) \leq \ell \}, 1 \leq \ell \leq n$. See Displays \eqref{convP} and \eqref{convQ}. The limiting functions can be expressed in terms of $\xi$ involving a random time-change.

\medskip It is natural to compare $L_n$ to the height $H_n$. In particular, since $L_n \leq H_n$, the bound \eqref{boundh} on the tail of $H_n$ also applies to the right tail of $L_n$. 
From \eqref{convH}, it follows that the limit law for the height of $\tau_n$ has very little mass at zero: 
$$\lim_{x \to 0} \lim_{n \to \infty} x^2 \log \Prob{H_n \sigma \leq x \sqrt{2n}}  =  \lim_{x\to 0} x^2 \log \Prob{\sqrt{2} \sup_{t \in [0,1]} \mathbf e(t) \leq x} = -\pi^2.$$ Our next result shows that the decay of the distribution function of $T_\infty$ is considerably slower. Still, all its derivatives vanish at 0.

\begin{prop} \label{heavyp}
 We have 
\begin{align*} \frac{1}{\log 2} & \leq \liminf_{x \to 0} \frac{- \log  \Prob{T_\infty \leq x}}{\log^2 x} \leq \limsup_{x \to 0} \frac{- \log  \Prob{T_\infty \leq x }}{\log^2 x} \leq \frac{2}{ \log 2}.
\end{align*}

\end{prop}

The proof of the proposition relies on sandwiching the random variable $L_n / \sqrt{n}$ between two quantities admitting series representations of the form $\sum_{i=0}^\infty  \rho^i Y_i$ for some $0 < \rho < 1$ and a sequence of independent and identically distributed random variables $Y_1, Y_2, \ldots$ It is presented in Appendix C. 

\section*{Acknowledgements}
The research of Luc Devroye was in part supported by an NSERC Discovery grant. The research of Cecilia Holmgren was 
partially supported by a grant of the Swedish Research Council which also allowed her visit at McGill University in November 2015 during which most of the work was carried out. The research of Henning Sulzbach was supported by a Feodor Lynen Research Fellowship of the Alexander von Humboldt Foundation.  The authors would also like to thank Louigi Addario-Berry and Christina Goldschmidt for valuable discussions in particular regarding the arguments involved in the study of the heavy path.

\setlength{\bibsep}{0.3em}
{\bibliographystyle{plainnat} 
\bibliography{gw}

\section*{Appendix A: proofs of Theorems \ref{thm:tailbounds} and \ref{thm:tailbounds2} } \label{apa}

From \eqref{ident22} and \eqref{LLTgen}, it follows that there exists $\omega_1> 0$ such that 
\begin{align} \label{eq1} \sup_{k > 0, k \in \mathbb N h -n} \frac 1 k  \Prob{|\mathcal T_1| + \ldots + |\mathcal T_k| = n} \leq \omega_1 n^{-3/2}. \end{align}
Similarly, there exist $n_5 \in \N$ and $\omega_2 > 0$ such that, for all $n \geq n_5$ and $k \leq \sqrt{n}$ with $n-k \in \N h$, 
\begin{align} \label{eq2}  \frac 1 k \Prob{|\mathcal T_1| + \ldots + |\mathcal T_k| = n} \geq \omega_2 n^{-3/2}. \end{align}

The following two lemmas provide the tools necessary to prove both theorems.

\begin{lem} \label{lem:gam}
For all $\ell, t, n \geq 1$ and $1 \leq k < \ell$, 
\begin{align*}
\Prob{|\mathcal T_1| \geq t, \ldots, |\mathcal T_\ell| \geq t, |\mathcal T_1| + \ldots + |\mathcal T_\ell| = n} \leq \frac{\omega_1^\ell 16^{\ell - 1}}{n^{3/2} t^{(\ell-1)/2}}, 
\end{align*}
and
\begin{align*}
\Prob{|\mathcal T_1| \geq t, \ldots, |\mathcal T_k| \geq t, |\mathcal T_1| + \ldots + |\mathcal T_\ell| = n} \leq \frac{\omega_1^{k+1} 16^{k }(\ell - k)}{n^{3/2} t^{(k-1)/2}} \frac{1}{\sqrt{\min(kt, \ell - k)}}.
\end{align*}
\end{lem}

\begin{lem} \label{lem:gam2}
There exists $c_1 > 0,  C_1 > 0$ satisfying the following property: for all $\ell \geq 1$, there exists $n_6 = n_6(\ell)$, such that, for all $n \geq n_6$ with $n-\ell \in \mathbb N h$, and $C_1 \leq t \leq n /  \ell - C_1$,
\begin{align*}
\Prob{|\mathcal T_1| \geq t, \ldots, |\mathcal T_\ell| \geq t, |\mathcal T_1| + \ldots + |\mathcal T_\ell| = n} \geq \frac{c_1^\ell }{n^{3/2} t^{(\ell-1)/2}}.
\end{align*}
Similarly, for $1 \leq k < \ell$, there exist constants $\tilde c_1, \tilde C_1$ (depending on $k, \ell$), such that, for $n$ sufficiently large with $n - \ell \in \N h$, and
$\tilde C_1 \leq t \leq n /  k - \tilde C_1$,
\begin{align*}
\Prob{|\mathcal T_1| \geq t, \ldots, |\mathcal T_k| \geq t, |\mathcal T_1| + \ldots + |\mathcal T_\ell| = n} \geq \frac{\tilde c_1}{n^{3/2} t^{(k-1)/2}}.
\end{align*}
\end{lem}
 
The two lemmas rely on the following simple result. 
 
 \begin{lem} \label{lem:easy} There exists a constant $\chi > 0$ such that, for all $a, b \geq 1, n\geq 2$ with $a + b \leq n$, we have
$$\frac{\chi}{n^{3/2} \sqrt{\max(a,b)}} \leq \sum_{k = a}^{n-b} (k(n-k))^{-3/2} \leq \frac{16}{n^{3/2} \sqrt{\min(a,b)}}$$
\end{lem}
\begin{proof} Both bounds follow easily from an application of the Euler-Maclaurin-formula using the symmetry of the sequence.
\end{proof}

\begin{proof} [Proof of Lemma \ref{lem:gam}]
Let $S_{\ell}  := \{(x_1, \ldots, x_{\ell}) : x_1, \ldots, x_{\ell} \geq t,  x_1 + \ldots + x_{\ell} \leq n-t\}$. Using \eqref{eq1}, we have 
\begin{align*}
& \Prob{|\mathcal T_1| \geq t, \ldots, |\mathcal T_\ell| \geq t, |\mathcal T_1| + \ldots + |\mathcal T_\ell| = n} \\
& = \sum_{(k_1, \ldots, k_{\ell -1}) \in S_{\ell - 1}} \Prob{|\mathcal T_1| = k_1, \ldots, |T_{\ell-1}| = k_{\ell-1}, |\mathcal T_\ell| = n - k_1 - \ldots - k_{\ell-1}} \\
& \leq \omega_1^\ell \sum_{(k_1, \ldots, k_{\ell -1}) \in S_{\ell  -1}} (k_1 \cdots k_{\ell-1}(n- k_1 - \ldots - k_{\ell -1}))^{-3/2}.
\end{align*}
Applying Lemma \ref{lem:easy} multiple times, 
\begin{align*}
& \sum_{(k_1, \ldots, k_{\ell -1}) \in S_{\ell  -1}} (k_1 \cdots k_{\ell-1}(n- k_1 - \ldots - k_{\ell -1}))^{-3/2} \\
& = \sum_{(k_1, \ldots, k_{\ell -2}) \in S_{\ell  -2}} (k_1 \cdots k_{\ell-2})^{-3/2} \sum_{t \leq k_{\ell - 1} \leq n -t - k_1 - \ldots  - k_{\ell - 2}} (k_{\ell - 1} (n- k_1 - \ldots - k_{\ell -1}))^{-3/2} \\
& \leq \frac{16}{\sqrt{t}} \sum_{(k_1, \ldots, k_{\ell -2}) \in S_{\ell  -2}} (k_1 \cdots k_{\ell-2}(n- k_1 - \ldots - k_{\ell -2}))^{-3/2} \\
& \leq \frac{16^{\ell -1}}{n^{3/2} t^{(\ell-1)/2}}.
\end{align*}
This shows the first inequality. Next, using the first inequality, \eqref{eq2}, and Lemma \ref{lem:easy},
\begin{align*}
& \Prob{|\mathcal T_1| \geq t, \ldots, |\mathcal T_k| \geq t, |\mathcal T_1| + \ldots + |\mathcal T_\ell| = n} \\
& = \sum_{j = kt}^{n-(\ell - k)} \Prob{|\mathcal T_1| \geq t, \ldots, |\mathcal T_k| \geq t, |\mathcal T_1| + \ldots + |\mathcal T_k| = j} \Prob{|\mathcal T_{k+1}| + \ldots + |\mathcal T_\ell| = n-j} \\
& \leq \frac{\omega_1^{k + 1} 16^{k-1}(\ell - k)}{ t^{(k-1)/2}} \sum_{j = kt}^{n-(\ell - k)} (j(n-j))^{-3/2} \\
& \leq \frac{\omega_1^{k + 1} 16^{k}(\ell - k)}{ t^{(k-1)/2} n^{3/2}}  \frac{1}{\sqrt{\min(kt, \ell - k)}}.
\end{align*}
This concludes the proof.
\end{proof}

\begin{proof} [Proof of Lemma \ref{lem:gam2}]

Let $n - \ell \in \mathbb N$ and  $n$ be sufficiently large. Upon choosing $C_1 \geq n_5$ large enough such that $I \cap \{C_1, C_1 + 1, \ldots \} = (\N h + 1) \cap \{C_1, C_1 + 1, \ldots \}$, for $C_1 < t < n / \ell - C_1$, we have the identity
\begin{align*}
& \Prob{|\mathcal T_1| \geq t, \ldots, |\mathcal T_\ell| \geq t, |\mathcal T_1| + \ldots + |\mathcal T_\ell| = n} \\
& = \sum_{(k_1, \ldots, k_{\ell -1}) \in S_{\ell - 1}, \newline k_i \in \mathbb Nh +1} \Prob{|\mathcal T_1| = k_1, \ldots, |\mathcal T_{\ell-1}| = k_{\ell-1}, |\mathcal T_\ell| = n - k_1 - \ldots - k_{\ell-1}} \\
& \geq \omega_2^\ell \sum_{(k_1, \ldots, k_{\ell -1}) \in S_{\ell  -1},  k_i \in \mathbb Nh +1} (k_1 \cdots k_{\ell-1}(n- k_1 - \ldots - k_{\ell -1}))^{-3/2}.
\end{align*}
Assume $h = 1$. Then, following the same lines as in the previous proof and using the lower bound in Lemma \ref{lem:easy}, we deduce that the right hand side is bounded from below by $(\omega_2 \chi)^\ell n^{-3/2} t^{(1-\ell)/2}$. The general case $h > 1$ follows similarly and we 
do not present the straightforward modifications. 
Next, with $C_1$ as before and $\tilde C_1 \leq t \leq n / k - \tilde C_1$ for sufficiently large $\tilde C_1$ (in particular, $\tilde C_1 > C_1$), and all $n$ sufficiently large,
\begin{align*}
& \Prob{|\mathcal T_1| \geq t, \ldots, |\mathcal T_k| \geq t, |\mathcal T_1| + \ldots + |\mathcal T_\ell| = n} \\
& \geq \sum_{j = (t+C_1)k, j-k \in \N h}^{n-(\ell - k)} \Prob{|\mathcal T_1| \geq t, \ldots, |\mathcal T_k| \geq t, \sum_{\ell=1}^k |\mathcal T_i|  = j} \Prob{|\mathcal T_{k+1}| + \ldots + |\mathcal T_\ell| = n-j} \\
& \geq \frac{c_1^{k}}{ t^{(k-1)/2}} \sum_{j = (t+C_1)k, j-k \in \N h}^{n-(\ell - k)} (j(n-j))^{-3/2}.
\end{align*}
Now, let $C_2 > \ell - k$ be minimal with $n - C_2 - k \in \mathbb N h$. Then, the right hand side is bounded from below by $C_2^{-3/2} c_1^{k}  t^{(1-k)/2} n^{-3/2}$. This concludes the proof.
\end{proof}

\begin{proof}[Proof of  Theorem \ref{thm:tailbounds}] We may assume $n \in I$ and $t \geq 1$. First, 
\begin{align*}
\Prob{N_k \geq t} \leq \sum_{\ell \geq k} p_\ell {\ell \choose k} \frac{\Prob{|\mathcal T_1| \geq t, \ldots, |\mathcal T_k| \geq t, |\mathcal T_1| + \ldots + |\mathcal T_\ell| = n-1}} {\Prob{|\mathcal T| = n }}.
\end{align*}
By Lemma \ref{lem:gam},
\begin{align*}
\Prob{N_k \geq t} \leq \frac{(16\omega_1)^{k+1}}{ \Prob{|\mathcal T| = n } t^{(k-1)/2} (n-1)^{3/2}} \left[1 + \sum_{\ell \geq k+1} \frac{p_\ell \ell^k (\ell - k)}{(\min(kt, \ell - k))^{1/2}} \right].
\end{align*}
Since $\E{\xi^{k+1}} < \infty$, the second factor in this display is bounded. Inequality \eqref{firsttail} now follows by approximating $\Prob{|\mathcal T| = n} $ with the help of \eqref{ident1} and \eqref{LLT}. 

To move from $N_k$ to $N_{k+}$, note that, for non-negative numbers $u_1, \ldots, u_n, t$, in order to have $u_1 + \ldots  + u_n \geq t$, we need to have $\max(u_1, \ldots, u_n) \geq t/n$. Thus, $\Prob{N_{k+} \geq t} \leq \Prob{N_k \geq t(\xi_\epsilon-k+1)^{-1}}.$ As above,
\begin{align*}
\Prob{N_{k+} \geq t} & \leq \Prob{\xi_\epsilon \geq t+k}  + (\Prob{|\mathcal T| = n })^{-1} \cdot \\
& \sum_{\ell = k}^{t+k-1} p_\ell {\ell \choose k} \Prob{|\mathcal T_1|Ê\geq \frac {t}{\ell - k + 1}, \ldots, |\mathcal T_k| \geq  \frac {t}{\ell - k + 1}, \sum_{j=1}^\ell |\mathcal T_i|  = n-1} . \end{align*}
The second summand is bounded from above by
\begin{align*}
\frac{(16\omega_1)^{k+1}}{ \Prob{|\mathcal T| = n } t^{(k-1)/2} (n-1)^{3/2}} \left[1 + \sum_{\ell = k+1}^{t + k -1}  \frac{p_\ell \ell^k (\ell - k+1)^{(k+1)/2}} {\min(kt/(\ell -k+1), \ell - k))^{1/2}} \right].
\end{align*}
Since $\E{\xi^{(3k+1)/2}} < \infty$, using the same ideas as above, the last term is at most of order $t^{(1-k)/2}$. Further, by Markov's inequality, using Proposition \ref{prop2},
$$ \Prob{\xi_\epsilon \geq t+k} \leq \frac{\E{\xi_\epsilon^{(k-1)/2}}}{(t+k)^{(k-1)/2}}  = O\left(\E{\xi^{(k+1)/2}} (t+k)^{(1-k)/2}\right).$$
The claim follows. 
\end{proof}
\begin{proof} [Proof of Theorem \ref{thm:tailbounds2}] Let $\ell \geq 2$ and $\lambda = \min \{i \geq \ell:p_i > 0\}$. Then,
\begin{align*} \Prob{N_\ell \geq t} \geq p_\lambda (\Prob{|\mathcal T| = n})^{-1} \Prob{|\mathcal T_1| \geq t, \ldots, |\mathcal T_\ell| \geq t, |\mathcal T_1| + \ldots + |\mathcal T_\lambda| = n-1}. \end{align*}
For sufficiently large $n$, under the conditions of the theorem, the right hand side is non-zero. The assertion follows immediately from Lemma \ref{lem:gam2}.
\end{proof}

It remains to verify the claim in the remark following Theorem \ref{thm:tailbounds2}, that is, for some $\varepsilon >0$, we have $\lim_{n \to \infty} \sup_{t \leq \varepsilon n} \Prob{N_k \geq t} t^{\frac{k-1}{2}} = \infty$ if $\E{\xi^k} = \infty$.
To this end, assume that $n$ is sufficiently large and $\varepsilon_1 n \leq t \leq \varepsilon_2 n$ for some (small) $0 < \varepsilon_1 < \varepsilon_2 < 1$ depending on the offspring distribution and $k$ but not on $n$. Furthermore, let $C > 0$ and $K$ chosen in such a way that
$\sum_{\ell = k + 1}^K p_\ell {\ell \choose k} \geq C$. We also suppose that $h=1$ for the sake of presentation. Then, using Lemma \ref{lem:gam2}, there exists $c >0$ such that
\begin{align*}
& \Prob{N_k \geq t} \Prob{|\mathcal T| = n} \\
&\geq \sum_{\ell = k+1}^K p_\ell {\ell \choose k} \Prob{|\mathcal T_1| \geq t, \ldots, |\mathcal T_k | \geq t, |\mathcal T_{k+1}| \leq \frac{t}{\ell - k}, \ldots,  |\mathcal T_{\ell}| \leq \frac{t}{\ell - k}, \sum_{j = 1}^\ell |\mathcal T_{j}|  =  n-1} \\
& =  \sum_{\ell = k+1}^K p_\ell {\ell \choose k} \sum_{\substack{0 \leq c_{k+1}, \ldots, c_\ell \\ \leq t / (\ell - k)}}  \mathbb P \bigg \{|\mathcal T_1| \geq t, \ldots, |\mathcal T_k | \geq t,  \\ & \hspace{5cm} \sum_{j= 1}^k |\mathcal T_{j}|  =  n-1 - \sum_{j = k+1}^\ell c_{j}\bigg \}  \prod_{m = k}^{\ell-1} \Prob{ |\mathcal T| = c_{m+1}} \\
& \geq c n^{-3/2} t^{(1-k)/2}  \sum_{\ell = k+1}^K p_\ell {\ell \choose k}  \sum_{0 \leq c_{k+1}, \ldots, c_\ell \leq t / (\ell - k)} \prod_{m = k+1}^\ell \Prob{ |\mathcal T| = c_m}  \\
& = c n^{-3/2} t^{(1-k)/2}   \sum_{\ell = k+1}^K p_\ell {\ell \choose k}   \left(\Prob{ |\mathcal T| \leq \frac{t}{\ell - k}} \right)^{\ell-k}.
\end{align*}
Using the well-known asymptotic expansion of $\Prob{|\mathcal T| = n}$, it follows that, for any sequence $t = t(n)$ with $\varepsilon_1 n \leq t \leq \varepsilon_2 n$, 
$$\liminf_{n \to \infty}  \Prob{N_k \geq t} t^{(k-1)/2} = c C \alpha^{-1}.$$
The assertion follows since $C$ was chosen arbitrarily.

\section*{Appendix B: proof of Proposition \ref{thm:general}}

We need the following result augmenting Corollary \ref{subtreebounds}.

\begin{lem} \label{lem:augment}
As $n \to \infty$, $n \in \N h + 1$, 
\begin{itemize}
\item [(i)] if $\E{\xi^3} < \infty$, then  $\E{\sqrt{N_2} \log^k (N_2 \vee 1)} = \Theta(\log^{k+1} n), k \in \N_0$,
\item [(ii)] if $\E{\xi^{7/2}} < \infty$, then $\E{\sqrt{N_{2+}} \log^k (N_{2+} \vee 1)} = \Theta(\log^{k+1} n), k \in \N_0$,
\item [(iii)] if $\E{\xi^{7/2}} < \infty$, then, for all $k \in \N_0$, there exist constants $\kappa_1^{(k)} \geq \kappa_2^{(k)} \geq 0$ such that, for all $n$ sufficiently large,  
$$-\kappa_1^{(k)} \log^k n \leq \E{\sqrt{N_1}\log^k N_1} - \sqrt{n} \log^k n \leq -\kappa_2^{(k)} \log^k n.$$
\end{itemize}
\end{lem}
\begin{proof}
(i) and (ii) follow as in the the proof of Corollary \ref{subtreebounds} using that the inverse function of $g(x) = \sqrt{x} \log^k(x \vee 1)$ is of the order $x^2 \log^{-2k}(x)$ as $x \to \infty$.
In order to prove (iii), note that, by (i), there exist constants $c_1, C_1 >0$ such that, for all $n \geq 1, n \in I$, 
$c_1 \sqrt{n-1} \leq \E{N_{2+}} \leq C_1 \sqrt{n-1}$ for all $n \geq 1$. By Jensen's inequality, for $n \geq 2, n \in I,$
\begin{align*}
\E{\sqrt{N_1}} \leq \sqrt{n-1 - \E{N_{2+}}} \leq \sqrt{n-1 - c_1\sqrt{n-1}} = \sqrt{n-1} - c_1 + O(n^{-1/2}).
\end{align*}
This shows the existence of $\kappa_2^{(0)}$. We may choose $\kappa_2^{(k)} = \kappa_2^{(0)}$ for all $k \geq 1$ since $N_1 \leq n$.  Next, 
\begin{align*}
\E{\sqrt{N_1}} = \sqrt{n-1} \E{\sqrt{1 - N_{2+} / (n-1)}} \geq \sqrt{n-1}\left(1 - \frac{\E{N_{2+}}}{n-1}\right) \geq \sqrt{n-1} - C_1.
\end{align*}
Thus, we can choose $\kappa_1^{(0)} = C_1$. We move on to the lower bound for $k\geq 1$. First, 
\begin{align*}
\E{\sqrt{N_1}\log^k N_1} & \geq \sqrt{n-1} \E{\log^k N_1} - \E{N_{2+} \log^k N_1}/\sqrt{n-1} \\ & \geq \sqrt{n-1} \E{\log^k N_1} - C_1 \log n.
\end{align*}
Hence, it is enough to show that $\E{\log^k N_1} = \log^k n + O(n^{-1/2} \log^k n)$.
To this end, for all $n$ sufficiently large, with $\beta_{2+}$ as in Theorem \ref{thm:tailbounds},
\begin{align*}
\E{\log^k N_1} &= \int_0^{\log^k (n-1)} \Prob{\log^k(n - 1- N_{2+}) \geq t} dt \\
& = \log^k(n-1) - k \int_{1}^{n-1} s^{-1} \log^{k-1}(s) \Prob{N_{2+} \geq n-1-s} ds \\
& \geq \log^k (n-1) - \beta_{2+} k  \int_{1}^{n-1} s^{-1} \log^{k-1}(s) (n-1-s)^{-1/2} ds \\
& = \log^k(n-1)- \beta_{2+} k (n-1)^{-1/2} \int_{1/(n-1)}^{1} t^{-1} (1-t)^{-1/2} \log^{k-1}(t(n-1)) dt \\
& \geq \log^k(n-1) -  \beta_{2+} k (n-1)^{-1/2} \Bigg(\sqrt{2} \int_{1/(n-1)}^{1} t^{-1} \log^{k-1}(t(n-1)) dt \\
& \hspace{5cm} + \log^{k-1} (n) \int_{0}^{1/2} t^{-1/2} (1-t)^{-1} dt\Bigg)  \\
& = \log^k(n-1) - \beta_{2+} (n-1)^{-1/2} ( \sqrt{2} \log^k (n-1) + 4 k \log^{k-1} (n) ).
\end{align*}
This concludes the proof.
\end{proof}

\begin{proof} [Proof of Proposition \ref{thm:general}]
Throughout the proof, keeping track of the size of $\tau_n$ in the notation, let $R_i(n) = | \mathcal V\left((1^*2)^i 1^*\right)|$ and $r_i(n) = \E{R_i(n)}$ for $i \geq 0$. Then, in distribution, 
$R_{i}(n) = R_{i}(N_1) + R_{i-1}(N_2) + 1$, where $(R_{i}(n))_{n \geq 0}, (R_{i-1}(n))_{n \geq 0}, (N_1, N_2)$ are independent. We prove the lower bounds by induction on $i \geq 1$ starting with $i = 1$. From Lemma \ref{lem:augment} (i), we know that $\E{r_0(N_2)} \geq c_1 \log n$ for some $c_1 > 0$ and all $n \geq 1$.
Assume that $r_1(k) \geq c \sqrt{k} \log k$ for all $k \leq n-1$ where $c \kappa_1^{(1)} < c_1$. Then, using Lemma \ref{lem:augment} (iii), 
\begin{align*}
r_1(n) & = \E{r_1(N_1)} + \E{r_0(N_2)} + 1 \geq c \E{\sqrt{N_1} \log N_1} + c_1 \log n \\
& \geq c \sqrt{n} \log n - (c \kappa_1^{(1)} -c_1) \log n  \\
& \geq c \sqrt{n} \log n.
\end{align*}
Therefore, $r_1(n) = \Omega( \sqrt{n} \log n)$.  The general proof runs along the same lines where we only need to replace $\E{r_{0}(N_2)} \geq c_1 \log n$ by $\E{r_{i-1}(N_2)} \geq c_i \log^i n$ for a suitable $c_i> 0$.
The upper bound can be proved by the same inductive argument.

We move on to (ii) and abbreviate $G_i(n) =| \mathcal V\left((1^*2)^i 1^*({3+}) \N^*\right)|$ and $g_i(n) = \E{G_i(n)}$. In distribution, $G_0(n) = G_0(N_1) + N_{3+} + 1$, where $(N_1, N_{3+}), (G_0(n))_{n \geq 0}$ are independent. Recall that $\E{N_{3+}} \leq C_1 \log n$ for some $ C_1 > 0$ and all $n \geq 1$. 
Assume that $g_0(k)  \leq C \sqrt{k} \log k$ for some $C > C_1  / \kappa_2^{(1)}$ and all $k \leq n-1$. Then, for $n$ sufficiently large, using Lemma \ref{lem:augment} (iii), 
$$g_0(n) \leq  \E{g_0(N_1)} + C_1 \log n  +1 \leq C \sqrt {n}  \log n + C_1 \log n - C \kappa_2^{(1)} \log n  + 1\leq C \sqrt{n} \log n.$$ The corresponding lower bound follows along the same lines.
For $i \geq 1$, we have the distributional recurrence $G_i(n) = G_i(N_1) + G_{i-1}(N_2) + 1$, where $(G_{i}(n))_{n \geq 0}, (G_{i-1}(n))_{n \geq 0}, (N_1, N_2)$ are independent, which is of the same form as the recurrence for $R_i(n)$. Thus, the same arguments as applied to $R_i(n)$ conclude the proof of (ii). The proof of (iii) runs along the same lines and is thus omitted.

Finally, let us consider $\mathcal V(\{2+\}^*)$. Let $L \geq 1$ (to be chosen later) and $\tau'$ be the subtree consisting of all nodes in  $\mathcal V(\{2+\}^*)$ with subtree size at most $L$. Then
$| V( \{2+\}^*) | \leq |\tau'| (1 + L)$.
Define $\varrho_n$ as the largest value of $k$ such that $N_k \geq L$. With $\mathcal T_1, \mathcal T_2, \ldots, \zeta$ as in Proposition \ref{prop2}, let $1 \leq \varrho \leq \zeta - 1$ be maximal with $|\mathcal T_{(\varrho : \zeta - 1)}| \geq L$. (Set $\varrho =0$ if $|\mathcal T_i | < L$ for all $1 \leq i \leq \zeta-1$.) Then, by Proposition \ref{prop2}, $\varrho_n \to \varrho + 1$ in distribution.
Since $\E{\xi_\epsilon^2} \to \E{\zeta^2}$ by Proposition \ref{prop2} and $\varrho_n \leq \xi_\epsilon$, we deduce $\E{\varrho_n} \to \E{\varrho} + 1$. 
Obviously, $\varrho \to 0$ in probability as $L \to \infty$. Again, since $\E{\xi^3} < \infty$, this convergence also holds in mean. Thus, upon choosing $L$ sufficiently large, we may assume that $q := \sup_{n \geq 1} \E{\varrho_n} < 2$. Now, in probability, the size of the tree $\tau'$ is bounded from above by the size of a branching process with offspring mean at most $q -1 $. Hence, $\E{|\tau'|}$ is uniformly bounded. 
\end{proof}

\section*{Appendix C: Proof of Proposition \ref{heavyp}}

The proofs uses the following lemma.

\begin{lem}
Let $Y_1, Y_2, \ldots$ be a sequence of non-negative, independent and identically distributed random variables with finite second moment. 
Let $Z_\rho = \sum_{i=0}^\infty \rho^i Y_i$ for $0 < \rho < 1$.
\begin{itemize}
\item [(i)] Assume that, for all $x \leq x_0$ and some $\varepsilon, \alpha > 0$, we have $\Prob{Y_1 \leq x} \geq \varepsilon x^\alpha$.
Then, 
$$\limsup_{x \to 0} \frac{- \log \Prob{Z_\rho \leq x}}{\log^2 x} \leq  \frac{\alpha}{\log 1 / \rho}.$$
\item [(ii)] Assume that, for all $x \leq x_0$ and some $\varepsilon, \alpha > 0$, we have $\Prob{Y_1 \leq x} \leq \varepsilon x^\alpha$.
Then, 
$$\liminf_{x \to 0} \frac{- \log \Prob{Z_\rho \leq x}}{\log^2 x} \geq  \frac{\alpha}{2 \log 1/ \rho}.$$
\end{itemize}

\end{lem}

\begin{proof}
We start with (i). Fix $0  < \rho < 1$, let $D > 0$ and $Z_D(t) = \sum_{i = - \lfloor D \log t  \rfloor }^\infty \rho^i Y_i$. Then 
$\E{Z_D(t)} = t^{- \lfloor D\log \rho  \rfloor} \E{Y_1} / (1-\rho)$ and $\text{Var}(Z_D(t)) = t^{-2  \lfloor D\log \rho  \rfloor} \text{Var}(Y_1) / (1-\rho^2)$. Thus, if $D > (- \log \rho)^{-1}$, then, by Chebyshev's inequality,
$$\Prob{Z_D(t) \leq t} \to 1, \quad t \to 0.$$
Next, 
\begin{align*}
\Prob{\sum_{i=0}^{-  \lfloor D \log t  \rfloor} \rho^i Y_i \leq t}  & \geq \Prob{Y_1 \leq \frac{t}{\sum_{i=0}^{-  \lfloor D \log t  \rfloor } \rho^i}}^{- \lfloor D \log t  \rfloor +1} \\
& \geq  \Prob{Y_1 \leq (t(1-\rho))}^{- \lfloor D \log t  \rfloor +1}  \\
& \geq (t(1-\rho))^{- \alpha ( \lfloor D \log t \rfloor  - 1)} \varepsilon^{- \lfloor D \log t \rfloor + 1}.
\end{align*}
Combining the two bounds proves (i). We move on to (ii). Let $0 < C < - (\log \rho)^{-1}$. Then, assuming $0 < x_0, \varepsilon < 1$, for all $t \leq x_0$, 
\begin{align*}
\Prob{Z_\rho \leq t}  \leq \Prob{\sum_{i=0}^{-\lfloor C \log t \rfloor } \rho^i Y_i \leq t} & \leq \prod_{i=0}^{-\lfloor C \log t \rfloor} \Prob{Y_1 \leq \rho^{-i} t } \\
& \leq \varepsilon^{-C \log t} t^{-C\alpha \log t} \prod_{i=0}^{-\lfloor C \log t \rfloor} \rho^{-i\alpha} \\
& \leq \varepsilon^{-C \log t} \exp \left(-C \alpha \log^2 t - \alpha \log \rho \frac{C^2 \log^2 t + C \log t} 2 \right). 
\end{align*}
By continuity, the inequality remains valid for $C = -(\log \rho)^{-1}$, and we choose this value to optimize the bound.
\end{proof}

\begin{proof} [Proof of Proposition \ref{heavyp} (Lower bound)]
Fix $0 < \delta< 1/2$ non-algebraic and $c := 1 / (1-\delta)$. (In particular, $(1-\delta)^i n \notin \N$ for all $i, n \geq 1$.) For $i \geq 0$, 
let $e_i \in \{1\}^*$ be the vector of length $i$ and $\sigma_i := \sigma_i(n) := \inf \{ j \geq 0:  N^*(e_j) \leq (1-\delta)^i n \}$. Then $L_n = \sigma_{\lceil \log_c n\rceil} - 1$.
The crucial observation is that there exist $C_1, C_2 > 0$ such that, for all $n \geq C_1$ and $j \leq  \lfloor \log_c n - C_2 \rfloor$, we have, stochastically, 
\begin{align}\label{sto1} \sigma_j \leq \sum_{i=1}^j G_i, \end{align}
where $G_1, G_2, \ldots$ is a sequence of independent geometrically distributed random variables on $\{1, 2, \ldots \}$ and $G_i$ has success parameter $\beta_2^* / \sqrt{\delta (1-\delta)^{i-1}n}$. Taking \eqref{sto1} for granted, we obtain, in a stochastic sense, 
$$  L_n = \sigma_{\lfloor \log_c n - C_2 \rfloor} + (L_n - \sigma_{\lfloor \log_c n - C_2 \rfloor}) \leq (1- \delta)^{-C_2-1} + \sum_{i=1}^{\lfloor \log_c n - C_2 \rfloor} G_i.$$
A simple direct computation using nothing but $1+x \leq e^x, x \in \R,$ shows that a geometrically distributed random variable with success probability $0 < p < 1$ is stochastically smaller than 
$1 + E / p$ where $E$ has the standard exponential distribution. It follows that, in probability, 
$$L_n \leq  \sum_{i=1}^{\lceil \log_c n \rceil } (1 + \sqrt{\delta n } (\beta_2^*)^{-1} (1-\delta)^{i/2} E_i) +(1- \delta)^{-C_2-1}, $$
where $E_1, E_2, \ldots$ is a sequence of independent random variables each of which having the standard exponential distribution. Hence, in probability, 
$$\frac{\beta_2^*(L_n  - (1- \delta)^{-C_2-1} -  \lceil  \log_c n \rceil )}{\sqrt{\delta n }} \leq  \sum_{i=0}^{\infty} (1-\delta)^{i/2} E_i.$$
It follows that $T_\infty \leq \frac  {\sigma \sqrt \delta}{2 \beta_2^*} \sum_{i=0}^{\infty} (1-\delta)^{i/2} E_i$ stochastically. From here, the lower bound on the limit inferior follows from the previous lemma. 

It remains to prove the bound \eqref{sto1}. Let $t \in \N$, $j \geq 0$ and $n \in I$. Then, 
\begin{align*}
& \Prob{\sigma_{j+1} \geq t} - \Prob{\sigma_j \geq t}  \\
& = \sum_{k=0}^{t-1} \sum_{\ell = \lceil (1-\delta)^{j+1} n \rceil}^{\lfloor (1-\delta)^{j} n \rfloor} \Prob{\sigma_{j+1}  \geq t | N^*(e_k) = \ell,\sigma_j = k } \Prob{N^*(e_k) = \ell,\sigma_j = k} \\
& =  \sum_{k=0}^{t-1} \sum_{\ell = \lceil (1-\delta)^{j+1} n \rceil}^{\lfloor (1-\delta)^{j} n \rfloor} \Prob{ \tilde N(\ell, t-k-1) >  (1-\delta)^{j+1} n }  \Prob{N^*(e_k) = \ell,\sigma_j = k}, 
\end{align*}
where $(\tilde N(\ell, i))_{i \geq 0}$ is distributed like $(N^*(e_i))_{i \geq 0}$ but in the tree $\tau_\ell$. For any $(1-\delta)^{j+1}n < m \leq \ell \leq \lfloor (1-\delta)^{j} n \rfloor$ and $i \geq 1$, we have
\begin{align*}
\Prob { \tilde N(\ell, i) \leq (1-\delta)^{j+1} n |  \tilde N(\ell, i-1) = m} = \Prob { \tilde N_{2+}(m) \geq m -  (1-\delta)^{j+1} n}.
\end{align*}
Now we specify $C_1$ and $C_2$ in order to apply Theorem \ref{thm:tailbounds2}. First, let $C_2$ be large enough such that $(1- \delta)^{1-C_2} \geq \max \{n_2, (2s_2+3)/(1-2\delta)\}$ with $n_2, s_2$ as in Theorem \ref{thm:tailbounds2}. Then, let $C_1 = c^{C_2 + 2}$. 
By Theorem \ref{thm:tailbounds2}, for all $n \geq C_1$ and $j \leq \lfloor \log_c n - C_2 \rfloor -1$, the right hand side of the last display is bounded from below by $\beta_2^* (m -  (1-\delta)^{j+1} n)^{-1/2} \geq \beta_2^* (\delta (1-\delta)^j n)^{-1/2}$.
Since $(\tilde N(\ell, i))_{i \geq 1}$ is a Markov chain, we have 
\begin{align*}
& \Prob{\sigma_{j+1} \geq t}  - \Prob{\sigma_j \geq t} \\
& \leq  \sum_{k=0}^{t-1} \sum_{\ell = \lceil (1-\delta)^{j+1} n \rceil}^{\lfloor (1-\delta)^{j} n \rfloor} \left(1- \frac{\beta_2^*}{\sqrt{ \delta(1-\delta)^j n}}\right)^{t-k-1} \Prob{N^*(e_k) = \ell,\sigma_j = k}  \\
& \leq   \sum_{k=0}^{t-1}  \left(1- \frac{\beta_2^*}{\sqrt{ \delta(1-\delta)^j n}}\right)^{t-k-1} \Prob{\sigma_j = k}.
\end{align*}
Hence, $\Prob{\sigma_{j+1} \geq t} \leq \Prob{\sigma_{j} + G_{j+1} \geq t}$ where $\sigma_j$ and $G_{j+1}$ are independent. Iterating the argument concludes the proof.
\end{proof}
\begin{proof} [Proof of Proposition \ref{heavyp} (Upper bound)]
First of all, since the scaling limit $T_\infty$ does not depend on the offspring distribution, we may assume that $p_0 = p_2 = 1/2$. In particular, $\sigma = 1$. 
Next, let $\{U_{i,j} : i, j \geq 1\}$ be a family of independent random variables with the uniform distribution on $[0,1]$. Let $2 < a' < a$ be non-algebraic. For $i \geq 1$, define
\begin{align*}
Q_i & = | \{j \geq 0: N^*(e_j) \in (n a^{-i}, n a^{-i+1}] \}, \\
R_i & = \min \left \{t \in \N : \sum_{j=1}^t \beta_{2}^2 U_{i,j}^{-2} \geq n (m_i - a^{-i})  \right\}, \quad m_i = \frac{a^{-i+1}}{a'}.
\end{align*}
Fix $k \in \N$ (large). We will show that for all $n$ sufficiently large, stochastically,  
\begin{align} \label{inq2}\sum_{i=1}^k Q_i \geq \sum_{i=1}^k R_i.\end{align}
For now, let us use this bound to conclude the proof of the proposition. 
Note that the random variable $U_{1,1}^{-2}$ is in the domain of attraction of a non-negative stable distribution with index $1/2$. More precisely, for some $c > 0$, 
$$n^{-2}\sum_{j=1}^n U_{1,j}^{-2} \stackrel{d}{\longrightarrow} \mathcal S, \quad \log \E{e^{i \lambda \mathcal S}} = - c |\lambda|^{1/2} (1-i \ \text{sign}(t)).$$
The limit law is the Levy distribution with density $\sqrt{c / (2\pi)} x^{-3/2} e^{-c / (2x)}$ on $[0,\infty)$. A straightforward computation shows that $\mathcal S^{-1/2}$ is distributed like $c^{-1/2} |\mathcal N|$, where $\mathcal N$ has the standard normal distribution.
In particular, for any $x > 0$, as $n \to \infty$, 
\begin{align*}
\Prob{R_i / \sqrt{n} \geq x} 
 \to \Prob{(m_i - a^{-i})^{1/2} (c \beta_2^2)^{-1/2}  |\mathcal N| \geq x}.
\end{align*}
It follows that, for $x > 0$, 
\begin{align*}
\Prob{  T_\infty \leq x / 2} = \lim_{n \to \infty} \Prob{L_n \leq x \sqrt{n}} & \leq \limsup_{n \to \infty} \Prob{\sum_{i=1}^k Q_i \leq x \sqrt{n}} \\
& \leq \lim_{n \to \infty} \Prob{\sum_{i=1}^k R_i \leq x \sqrt{n}}  \\ 
& = \Prob{\sum_{i=1}^k (c\beta_2^2)^{-1/2} (m_i - a^{-i})^{1/2} |\mathcal N_i| \leq x} \\
& = \Prob{ (c\beta_2^2)^{-1/2} (a/a' - 1)^{1/2} \sum_{i=1}^k a^{-i/2} |\mathcal N_i| \leq x},
\end{align*}
where $\mathcal N_1, \mathcal N_2, \ldots$ are independent standard normal random variables.
Since the left hand side does not depend on $k$, we may substitute $k = \infty$ on the right hand side. The previous lemma concludes the proof since we can choose $a > 2$ arbitrarily.

It remains to prove \eqref{inq2}. To this end, for $i \geq 1$, define  $P_i  = \max \{N(j) : N(j) \in [n m_i, n a^{-i+1}] \}.$ Subsequently, assume that $n \geq 4 a^k a' / (a'-2)$.
Then, since for all non-leafs $v \in [n]$, we have $N^*(\kappa(v)1) \geq (N(v)-1) / 2$, a simple computation shows that the quantities $P_1, \ldots, P_k$ are well-defined.
Let $t > 0$. 
Then, 
\begin{align*}
\Prob{\sum_{i=1}^k Q_i \geq t} & = \sum^{\lfloor n a^{-k+1}\rfloor}_{x = \lceil n m_k \rceil} \Prob{\sum_{i=1}^k Q_i \geq t, P_k = x} \\
& = \sum^{\lfloor n a^{-k+1}\rfloor}_{x = \lceil n m_k \rceil} \sum_{\ell \geq 0} \Prob{Q_k \geq t - \ell, \sum_{i=1}^{k-1} Q_i = \ell \bigg | P_k = x} \Prob{P_k = x}. 
\end{align*}
Observe that, conditionally on $P_k = x$, the random variables $(Q_1, \ldots, Q_{k-1}), Q_k$ are independent. Hence, 
\begin{align*}
\Prob{\sum_{i=1}^k L_i \geq t} & = \sum_{\ell \geq 0}  \sum^{\lfloor n a^{-k+1}\rfloor}_{x = \lceil n m_k \rceil} \Prob{Q_k \geq t - \ell \bigg | P_k = x} \Prob{ \sum_{i=1}^{k-1} Q_i = \ell \bigg | P_k = x} \Prob{P_k = x}. 
\end{align*}
The crucial observation is that,  conditionally on $P_k = x$, the random variable $Q_k$ is stochastically larger than $R_k$. To see this, note that, by Theorem \ref{thm:tailbounds}, we know that $N_{2} \geq \beta_2^2 U_{1,1}^{-2}$ in probability. Hence, for any $ \lceil n m_k \rceil \leq x \leq \lfloor n a^{-k+1}\rfloor$ and $y \geq 1$, using the notation from the previous proof, we deduce 
\begin{align*} \Prob{Q_k \geq y | P_k = x} & =  \Prob{\tilde N (x, y) > n a^{-k}} \\
& \geq  \Prob{\sum_{j = 1}^{y - 1} \beta_2^2 U_{1,j}^{-2}< n(m_k - a^{-k})} = \Prob{R_k \geq y}.
\end{align*}
We conclude
\begin{align*}
\Prob{\sum_{i=1}^k Q_i \geq t} & \geq \sum_{\ell \geq 0} \Prob{R_k \geq t - \ell}  \sum^{\lfloor n a^{-k+1}\rfloor}_{x = \lceil n m_k \rceil} \Prob{ \sum_{i=1}^{k-1} Q_i = \ell \bigg | P_k = x} \Prob{P_k = x} \\
& = \sum_{\ell \geq 0} \Prob{R_k \geq t - \ell}  \Prob{ \sum_{i=1}^{k-1} Q_i = \ell} \\ 
& = \Prob{ \sum_{i=1}^{k-1} Q_i + R_k \geq t}. \end{align*}
Iterating gives the desired claim and finishes the proof.
\end{proof}

\section*{Appendix D: Proof of Theorem \ref{conv:hp}} \label{apb}

To keep this section self-contained, let us recall some definitions. For any discrete ordered rooted tree $\mathbb T$,  the heavy path is defined as the unique path from the root to a leaf which always continues in the largest subtree. Here, ties are 
broken considering the preorder index. It is easy to read off the length of the heavy path from the depth-first search process encoding $\mathbb T$ since each excursion above a level corresponds to 
a subtree. Thus, starting with the interval $I_0 := [0, 2|\mathbb T| - 2]$ at time $0$, given the interval $I_i$ at time $i \geq 0$, 
$I_{i+1}$ is chosen as the largest subinterval of $I_i$ corresponding to an excursion above level $i+1$. We now extend the concept to arbitrary continuous excursions.
To this end, let
$$\Cex := \{ f : [0,1] \to \R^+_0 \text{ continuous}: f(0) = f(1) = 1\}.$$
 We always consider $\Cex$ endowed with the topology induced by the supremum norm $\| f \| = \sup_{t \in [0,1]} |f(t)|$. 

\medskip \noindent \textbf{Superlevel sets for excursions.} 
Let $\mathcal V$ be the space of open subsets of $[0,1]$, where open refers to the subspace topology of $[0,1]$ in $\R$. For $O_1, O_2 \in \mathcal V$, we define $d(O_1, O_2) = d_{\mathrm H}(O_1^c, O_2^c)$, where $d_{\mathrm H}$ denotes the Hausdorff distance. 
For $O \in \mathcal V$ and a $\mathcal V$-valued sequence $O_n, n \geq 0$, we have $d(O_n, O) \to 0$ if any only if $\lambda(O_n \Delta O) \to 0$ where $A \Delta B := A \backslash B \cup B \backslash A$ and $\lambda$ denotes the Lebesgue measure on $[0,1]$.  $(\mathcal V, d)$ is a compact metric space (hence Polish). Every element of $\mathcal V$ uniquely decomposes in at most countably many disjoint open intervals.  

For a function $f \in \Cex$ and $t \geq 0$, the superlevel set $\mathcal P_f(t) = \{s \in [0,1]: f(s) > t\}$ is open. The $\mathcal V$-valued process $\mathcal P_f := \mathcal P_f(t), t \geq 0$ 
has the following properties
\begin{enumerate}
\item [(i)] $\mathcal P_f(t) \subseteq \mathcal P_f(s)$ for $0 \leq s \leq t$, 
\item [(ii)]  $\mathcal P_f$ is right-continuous, that is, $\mathcal P_f(t) = \mathcal P_f^+(t)  := \lim_{s \downarrow t} \mathcal P_f(s)$ for all $t \geq 0$, 
\item [(iii)] $\mathcal P_f(t) = \emptyset$ for all $t$ large enough, and 
\item [(iv)] $x \in \partial \mathcal P_f(t) \Rightarrow x \notin \partial \mathcal P_f(s)$ for all $0 \leq t < s$. 
\end{enumerate}
Here, and subsequently, $\partial O$ denotes the boundary of an open set $O \subseteq [0,1]$. Conversely, for every $\mathcal V$-valued process $\mathcal P_t, t \geq 0$ satisfying (i)--(iii), we can define
$$f_{\mathcal P}(t) = \sup \{s \geq 0: t \in \mathcal P_s\},$$
and observe that $\mathcal P_t = \mathcal P_{f_{\mathcal P}}(t)$ for all $t \geq 0$.
Note that $f_{\mathcal P}$ is lower semi-continuous. (A non-negative function on $[0,1]$ is lower semi-continuous if and only if $\mathcal P_f(t)$ is open for all $t \geq 0$. 
Lower semi-continuous functions are the most natural class of functions in the context of tree encodings. See, e.g.\ Duquesne \cite{duq} for a complete characterization.) Further, 
$f_{\mathcal P} \in \Cex$ if and only $\mathcal P_t, t \geq 0$ satisfies (iv). In particular, letting $\mathcal W$ be set of $\mathcal V$-valued processes satisfying (i)--(iv), the map 
$f \mapsto \mathcal P_f$ is a bijection between $\Cex$ and $\mathcal W$.

\medskip \noindent \textbf{The heavy path construction.}  For $O \in \mathcal V$, let $\mathfrak m(O)$ denote the interval with largest length in $O$. In case several intervals qualify, we choose the smallest of them with respect to the order $\preceq$ defined for intervals $I, I'$ by 
$$I \preceq I' :\Leftrightarrow \inf I \leq \inf I'.$$
For a process $\mathcal P$,  we define a process $\mathcal P^*_t, t \geq 0$ with $\mathcal P^*_t \subseteq \mathcal P_t$ for all $t \geq 0$ as follows: set $\mathcal P^*_0 = \mathcal P_0$ and $T_0 = 0$. 
Then, inductively, for $n \geq 0$, given $T_n$ and $\mathcal P^*_t$ for all $t \leq T_n$, let 
\begin{align*}
T_{n+1} &= \inf \{ t > T_n: \mathfrak m(\mathcal P^*_{T_n} \cap \mathcal P_t) \leq  2^{-(n+1)} \}, \\ 
 \mathcal P^*_t& = \mathfrak m ( \mathcal P^*_{T_n} \cap \mathcal P_t), \quad T_n < t <T_{n+1}, \\
 \mathcal P^*_{T_{n+1}} & = \mathfrak m \left( \lim_{s \uparrow T_{n+1}}\mathcal P^*_s \cap \mathcal P_t \right). \end{align*}
 $T_\infty := \lim_{n \to \infty} T_n$ is finite and bounded by $ \inf \{t \geq 0: \mathcal P_t = \emptyset \}$.  For $t \geq T_\infty$, we set $\mathcal P^*_t = \emptyset$. Then, 
 $\mathcal P^* \in \mathcal W$ and $\mathcal P^*_t$ is an interval for all $t \geq 0$. We also define $t_* = \lim_{n \to \infty} \inf \mathcal P^*_{T_n}$ and $t^* = \lim_{n \to \infty} \sup \mathcal P^*_{T_n}$.
We call $\mathcal P$ trivial if $\mathcal P_t = \emptyset$ for all $t \geq 0$. For a non-trivial process $\mathcal P_t, t \geq 0$, two scenarios are possible:
\begin{itemize}
\item [(i)] $T_n < T_\infty$ for all $n \geq 1$. Then, $\mathcal P^*_t$ is continuous at $T_\infty$ and $t_* = t^*$.
\item [(ii)]  $T_n = T_\infty$ for some $n \geq 1$. Then, $\mathcal P^*_t$ is discontinuous at $T_\infty$ and  $t_* < t^*$.  \end{itemize}
For $f \in \Cex$, write $\mathcal P^*_f$ for $\mathcal P^*$ and $T^f_\infty$ for $T_\infty$ when $\mathcal P = \mathcal P_f$. 
If $f$ is the depth-first search process of a discrete ordered rooted tree rescaled on the unit interval then $T^f_\infty$ is the length of the corresponding heavy path.

\medskip \textbf{Remark.} The sequence $T_n, n \geq 0$ arising in the heavy path construction plays no role in the sequel. We could replace the sequence $2^{-(n+1)}, n \geq 0$ in its definition by any monotonically 
decreasing sequence $\alpha_n, n \geq 0$ with $\alpha_n \to 0$ and $\alpha_n \geq 2^{-(n+1)}$. This leaves $\mathcal P^*$ and $T_\infty$ invariant. In fact, we could also let $\alpha_n$ 
depend on $\mathcal P$ by setting $\alpha_n = \frac 1 2 \lambda(\mathcal P^*_{T_n})$.

\medskip \textbf{Remark.} The Brownian Continuum Random tree is obtained from a Brownian excursion $\mathbf e$ as the quotient space $[0,1] / \sim$ where $x \sim y$ if $\mathbf e(x) = \mathbf e(y) \geq \mathbf e(s)$ for all $s \in [x,y]$. 
In the standard construction, the limiting object becomes a compact measured rooted metric space, a so-called real tree, see \cite{evans, legallrandomtrees}. One could develop the heavy path theory more abstractly for real trees without relying on encodings by continuous functions, but there is no need for this generalization in our
work.

\medskip Unfortunately, some technical issues arise in this construction. The map $O \to \lambda (\mathfrak m(O))$ is continuous, and so is $(O, O') \mapsto  O \cap O'$. Similarly, the 
map $O \mapsto \inf O$ ($O \mapsto \sup O$, respectively) is measurable and continuous at $O \in \mathcal V$ if and only if $0 \in O$ ($1 \in O$, respectively). The map $O \to \mathfrak m(O)$ is measurable and continuous at $O \in \mathcal V$ if only if the largest interval in $O$ is 
unique. For any fixed $t \geq 0$, the map $f \to \mathcal P_f(t)$ is not continuous on $\Cex$. The set $\mathcal W$ is not closed when endowing the set of all $\mathcal V$-valued processes with the topology of uniform convergence on compact sets. 
The following important lemma contains a positive result in the converse direction. Here and subsequently, we recall the definition of the modulus of continuity of a continuous function $f$ on $[0,1]$:
$$\omega_f(\varepsilon) = \sup_{|s-t| \leq \varepsilon} |f(t)  - f(s)|, \quad \varepsilon >0.$$
By the Arzela-Ascoli theorem, for a family of continuous functions $(f_i)$ on $[0,1]$, we have $\sup_{i} \omega_{f_i}(\varepsilon) \to 0$ as $\varepsilon \to 0$ if $(f_i)$ is relatively compact. (In other words, the family
is uniformly equicontinuous.)
\begin{lem} \label{lem_conv_one}
Let $f,f_n, n \geq 1$ be continuous excursions. Suppose that, uniformly on compact sets, we have $d((\mathcal P_{f_n}(t), \mathcal P_f(t)) \to 0$. Then, $\|f_n - f\| \to 0$.  
\end{lem}

\begin{proof}
 For ease of notation, abbreviate $\mathcal P_n := \mathcal P_{f_n}, n \geq 1$ and $\mathcal P := \mathcal P_{f}$. Fix $\varepsilon > 0$ and $n$ large enough such that 
 $d((\mathcal P_n)_t, \mathcal P_t) \leq \varepsilon$ for all $0 \leq t \leq \|f\|$. Fix $t \in [0,1]$ and let $x_n = f_n(t)$. Suppose that $t \in \partial \mathcal P_n(x_n)$. Then, 
 there exists $t'_n \in f^{-1}(\{x_n\})$ with $|t'_n - t| \leq \varepsilon$ and $t'_n \in \partial \mathcal P(x_n)$. This implies $|f(t) - f_n(t)| \leq \omega_f(\varepsilon)$. If $t \notin \partial \mathcal P_n(x_n)$, then 
 $f_n = x_n$ on some closed interval $I_n$ containing $t$ which may choose maximal. 
 If $\sup I_n < t + 2 \varepsilon$, then, since $\sup I_n \in \partial \mathcal P_n(x_n)$, we have $|f_n(t) - f(t)| \leq |f_n(\sup I_n) - f(\sup I_n)| + |f(\sup I_n) - f(t)| \leq 2 \omega_f(2\varepsilon)$ from the first part of the proof. 
 The same bound follows if $\inf I > t - 2 \varepsilon$.
 Now, assume $[t-2\varepsilon, t + 2 \varepsilon] \subseteq I$. Then, we must have $f \geq x_n$ on $[t-\varepsilon, t + \varepsilon]$. 
 If $f(t) \neq f_n(t)$, since $f > x_n$ is not possible on the entire interval $[t-2\varepsilon, t +2\varepsilon]$, there exists $t_n' \in [t-2\varepsilon, t + 2 \varepsilon]$ with $t_n' \in \partial \mathcal P(x_n)$. 
 As above, this implies
 $|f(t) - f_n(t)| \leq  \omega_f(2\varepsilon)$.
  Since $f$ is continuous, we have $\omega_f(\varepsilon) \to 0$ as $\varepsilon \to 0$ finishing the proof.
  \end{proof}

\medskip \noindent \textbf{The Skorokhod space.} Let $(S,d)$ be a Polish space. By $\mathcal D_S$ we denote the set of \cadlag \ functions with values in $S$.  A function $f: [0,\infty) \to S$ is called \cadlag \ if, for all $t \geq 0$,  
it is right-continuous at $t$ and, for all $t > 0$, the left limit $f(t-) := \lim_{s \uparrow t} f(s)$ exists.  
For every $f \in \mathcal D_S$, the set of 
discontinuities $\{t \in [0, \infty) : f(t) \neq f(t-) \}$ is at most countable. 
$\mathcal D_S$ is endowed with the Skorokhod topology: a sequence $f_n, n \geq 1$ converges to a function $f$ if and only if there exists a sequence of strictly increasing continuous functions 
$\lambda_n : [0, \infty)  \to  [0, \infty)$ such that $\lambda_n \to \text{id}$ uniformly on $[0, \infty)$ and $f_n \circ \lambda_n \to f$ uniformly on compact sets. 
If $f_n \to f$ in the Skorokhod topology, and $f$ is continuous at $t \in [0, \infty)$, then
$f_n(t) \to f(t)$. $\mathcal D_S$ is a Polish space, and the Borel-$\sigma$-algebra is generated by the family of projections $\pi_t: \mathcal D_S \to S, \pi_t(f) = f(t), t \geq 0$. All these properties and more information on $\mathcal D_S$
can be found in Billingsley's book \cite{bil68}.
Again, one can easily check that $f \mapsto \mathcal P_{f} $ is not continuous on $\Cex$. Further, $\mathcal W \subseteq \mathcal D_{\mathcal V}$ is not closed. 
($\partial \mathcal{W}$ contains processes generated by lower semi-continuous functions which are not even \cadlag.) The following lemma is crucial.
\begin{lem}
The set $\mathcal W \subseteq \mathcal D_{\mathcal V}$ endowed with its relative topology is Polish. In particular, $\mathcal W$ is measurable with respect to the Borel-$\sigma$-algebra on $\mathcal D_{\mathcal V}$. 
Also, the map 
$f \mapsto \mathcal P_f$ from $\Cex$ to $\mathcal D_{\mathcal V}$ is measurable.
\end{lem}

\begin{proof}
Let us first show that $\mathcal P \mapsto f_{\mathcal P}$ is continuous regarded as map $\mathcal W \to \Cex$. To this end, 
let $\mathcal P, \mathcal P_n, n \geq 1$ be elements in $\mathcal W$ with $\mathcal P_n \to \mathcal P$ in the Skorokhod topology. Choose a sequence $\lambda_n, n \geq 1$ of strictly increasing continuous bijections on $[0,\infty)$ with
$\lambda_n \to \text{id}$ uniformly on $[0,\infty)$ and $\mathcal P_n \circ \lambda_n \to \mathcal P$ uniformly on compact sets. By Lemma \ref{lem_conv_one}, $ \|f_{\mathcal P_n \circ \lambda_n} - f_{\mathcal  P}\| \to 0$. Hence,  it remains to show that
$\|f_{\mathcal P_n \circ \lambda_n} - f_{\mathcal P_n}\| \to 0$. But for any $\mathcal P' \in \mathcal W$ and any strictly increasing bijection $\lambda$, we have $f_{\mathcal P' \circ \lambda} = \lambda^{-1} \circ f_{\mathcal P'}$. Thus, 
$\|f_{\mathcal P_n \circ \lambda_n} - f_{\mathcal P_n}\| \leq \sup_{t > 0} |\lambda(t) - t| \to 0$. This shows the claimed continuity.

In view of Lemma \ref{lem_conv_one}, for $\mathcal P, \mathcal P' \in \mathcal W$, define $$d^*(\mathcal P, \mathcal P') = \| f_{\mathcal P} - f_{\mathcal P'} \| + d_{\text{sk}}(\mathcal P, \mathcal P'),$$
where $d_{\text{sk}}$ denotes any complete metric generating the Skorokhod topology on $\mathcal D_{\mathcal V}$. (See \cite{bil68} for an explicit construction.)  Since $(\mathcal D_{\mathcal V}, d_{\text{sk}})$ 
is separable, the same follows for $(\mathcal W, d_{\text{sk}})$. From the continuity of $\mathcal P \mapsto f_{\mathcal P}$ it follows that $(\mathcal W, d^*)$ is separable.
If $\mathcal P_n, n \geq 1$ is Cauchy with respect to $d^*$, then it is Cauchy with respect to $d_{\text{sk}}$. Hence, there exists a $d_{\text{sk}}$-limit $\mathcal P' \in \mathcal D_{\mathcal V}$. Further, 
by definition and completeness of the supremum norm, there exists $g \in \Cex$ with
$\| f_{\mathcal P} - g \| \to 0$. Clearly, this implies $g = f_{\mathcal P'}$. Hence, $\mathcal W$ is complete with respect to $d^*$. 
By construction, the embedding $\mathcal W \to \mathcal D_{\mathcal V}$ is continuous. 
Both measurability of $\mathcal W$ and measurability of $f \mapsto \mathcal P_f$ now follow from the Lusin-Suslin theorem \cite[Theorem 15.1]{kechris}.\end{proof}

Finally, one also has to verify measurability of the quantities arising in the construction of the heavy path.

\begin{lem}
 The maps $\mathcal P \mapsto T_\infty$ and $\mathcal P \to \mathcal P^*$ are measurable.
\end{lem}

\begin{proof}
 We keep track of more quantities in the construction. Set $\mathcal P_0^{(0)} = \mathcal P_0$.
 Inductively, for $n \geq 0$, 
 \begin{align*}
  \mathcal P^{(n)}_t& := \mathfrak m ( \mathcal P^{(n)}_{T_n}) , \quad t < T_n, \\
 \mathcal P^{(n)}_t& := \mathfrak m ( \mathcal P^{(n)}_{T_n} \cap \mathcal P_t), \quad t > T_n, \\
 T_{n+1} &= \inf \{ t > 0: \lambda(\mathcal P^{(n)}_t) \leq  2^{-(n+1)} \}, \\ 
 \mathcal P^{(n+1)}_{T_{n+1}} & := \mathfrak m \left( \mathcal P^{(n)}_{T_{n+1}-} \cap \mathcal P_t \right). \end{align*}
(The third line is merely an observation.)
 Then, $\mathcal P_t^* = \sum_{i=0}^\infty \mathbf 1_{[T_i, T_{i+1})}(t) \mathcal P_t^{(i)}$.
In order to show that $\mathcal P \to \mathcal P^*$ is measurable, we need to verify that, for all $i \geq 0$,  $\mathcal P \to \mathcal P^{(i)}$ is measurable and that
$T_i$ is a stopping-time with respect to the family of $\sigma$-algebras $\mathcal F_t = \sigma ( \{\pi_s : 0 \leq s \leq t \})$. (This means that $\{T_i \leq t\} \in \mathcal F_t$ for all $t \geq 0$.)
This can be done by induction on $i$. Clearly, $\mathcal P_t^{(0)}$ is measurable. $T_1$ is a hitting-time of a closed set, therefore a stopping time by standard arguments.  Further, it is well-known that
$\mathcal P \to \mathcal P_T$ is measurable for any stopping-time $T$. Finally, the map $\mathcal P \mapsto \mathcal P_- := (\mathcal P_{t-}), t \geq 0$ is measurable. Hence, $\mathcal P_t^{(1)}$ is measurable. Now, proceed  inductively. Measurability of 
$\mathcal P \mapsto T_\infty$ follows since $T_\infty$ is the limit of measurable functions.
 \end{proof}

\medskip \noindent \textbf{Continuity properties.}
For $f \in \Cex$, define
$$M_f(x) = \{(s,t) : 0 \leq s < t \leq 1, f(s) = f(t) = x, f > x \text{ on } (s,t) \}, \quad x \geq 0.$$
Now, let
\begin{align*}
 \Cex^{(1)} =  \{ f \in \Cex : & \text{ For all } 0 \leq x \leq \|f \| \text{ there exists at most one pair } (s,t) \in M_f(x) \\ &  \text{ maximizing } t-s \}
\end{align*}
and
\begin{align*}
 \Cex^{(2)} = \{ f \in \Cex : & \text{ For all } t \geq 0 \text{ there exists at most one value } x \in (\mathcal P^*_f(t))^{\mathrm{o}} \\ & \text{ with } f(x) = f(\inf \mathcal P^*_f(t)) \}.
\end{align*}

In a Brownian excursion, all local minima are strict and pairwise distinct. Hence, for all $x \geq 0$, the set $M_x^f$ contains at most two elements and $\mathbf e \in \Cex^{(2)}$. It is well-known that every local minima $t$ 
does not decompose the interval $(\sup \{s< t : \mathbf e(s) > \mathbf e(t)\}, \inf \{ s > t : \mathbf e(s) > \mathbf e(t)\})$ equidistantly. Hence, $\mathbf e \in \Cex^{(1)}$.
For $f \in \Cex$, define
$$\mathbf m_f(t) := \lambda(\mathcal P^*_f(t)), \quad t \geq 0, \quad \zeta_f(t) := \inf \{s > 0: \mathbf m_f(s) \leq t\}, \quad t \in [0,1]. $$
 The map $t \mapsto \zeta_f(t)$ is continuous. Every point of discontinuity of $\mathcal P_f^*$ (or, equivalently, of $\mathbf m_f$) corresponds to an interval on which
$\zeta_f$ is constant.
For $f \in \Cex$  let $M_f = \{\mathbf m_f(t) : t  \geq 0 \}$. 
Further, for $r \geq 0$ and $f \in \Cex$, set
$$f^*_r(t) := (f(t) - f(\inf \mathcal P^*_f(r)) \mathbf 1 _{\mathcal P^*_f(r)}(t).$$
Clearly, if $f \in \Cex^{(1)}$ then $f^*_r \in \Cex^{(1)}$, analogously for $\Cex^{(2)}$. We now set $\Cex^* = \Cex^{(1)} \cap \Cex^{(2)}$.

In the following lemma, recall that, for a \cadlag\ function $f$ with values in a Polish space and $t > 0$, we have set $f(t-) := \lim_{s \uparrow t} f(s)$.
\begin{lem} \label{lem_cont}
 Let $f_n, n \geq 1$ be a sequence of continuous excursions and $f \in \Cex^*$. Suppose that $\|f_n - f\| \to 0$. Let
 $r \in M_f$ with $\mathbf m_f(\zeta_f(r)-) \geq \frac 1 2 \mathbf m_f(0)$. Then, there exists a sequence $r_n \to r$ with $\zeta_{f_n}(r_n) \to \zeta_{f}(r)$ such that
 \begin{align} \label{rightl} d( \mathcal P^*_{f_n}(\zeta_{f_n}(r_n)), \mathcal P^*_{f}(\zeta_f(r)) \to 0, \end{align}
 \begin{align} \label{leftl} d( \mathcal P^*_{f_n}(\zeta_{f_n}(r_n)-), \mathcal P^*_{f}(\zeta_f(r)-)) \to 0, \end{align}
 and 
 $\|(f_n)^*_{\zeta_{f_n}(r_n)} - f^*_{\zeta_f(r)} \| \to 0.$
 
\end{lem}

 \begin{proof}
It is easy to see that, for any $r,s \in [0,1]$ and $f, g \in \Cex$, we have
\begin{align*}
  | \| f^*_r & - g^*_s\|  -  \| f -g\| | \leq \\ 
 &  \omega_{f}( |\inf \mathcal P_{f}(\zeta_f(r)) - \inf \mathcal P_{g}(\zeta_g(s)| ) + 
\omega_{f}( |\sup \mathcal P_{f}(\zeta_f(r)) - \sup \mathcal P_{g}(\zeta_g(s))| ) \\ 
& + \omega_{g}( |\inf \mathcal P_{f}(\zeta_f(r)) - \inf \mathcal P_{g}(\zeta_g(s))| ) + 
\omega_{g}( |\sup \mathcal P_{f}(\zeta_f(r)) - \sup \mathcal P_{g}(\zeta_g(s))| ).
\end{align*}
Hence, the final claim of the lemma is a direct implication of the remaining statements.
If $\mathbf m_f$ is continuous at $\zeta_f(r)$, then we can simply choose $r_n = r$. In this case, if $r > \mathbf m_f(1/2)$, the assertions \eqref{rightl}, \eqref{leftl} even hold
for general $f \in \Cex$. The interesting case is when $\mathbf m_f$ is discontinuous at $\zeta_f(r)$ which we assume from now on. Let $\alpha = \inf  \mathcal P_f(\zeta_f(r)-)$ and
$\beta = \sup \mathcal P_f(\zeta_f(r)-)$. Since $f \in \Cex^{**}$ there exists a unique strict minimum $x$ of $f$ on $(\alpha, \beta)$ such that, either,  
i) $ \mathcal P_f(\zeta_f(r)) = (\alpha, x)$, or, ii) $ \mathcal P_f(\zeta_f(r)) = (x, \beta)$. We have $x \neq (\alpha + \beta)/2$ since $x \in \Cex^{*}$. Let 
$\alpha' = (\alpha+x)/2, \beta' = (\beta + x)/2$ and $s_n = \inf \{f_n(s) : \alpha' < s < \beta' \}$. In case of
i), let $x_n = \inf \{\alpha' < y < \beta' : f(y) = s_n \}$, while, for ii), we set $x_n = \sup \{\alpha' < y < \beta' : f(y) = s_n \}$.  Now, let $r_n = \mathbf m_{f_n}(s_n)$. Then, for all $n$ sufficiently large, 
there exist $\alpha_n< x_n < \beta_n$ such that $\mathcal P_{f_n}(s_n-) = (\alpha_n, \beta_n)$ and, for i), $\mathcal P_{f_n}(s_n) = (\alpha_n, x_n)$ while, for ii), $\mathcal P_{f_n}(s_n) = (x_n, \beta_n)$. We also have
$\alpha_n \to \alpha, \beta_n \to \beta$ and $x_n \to x$. All statements follow readily. 
\end{proof}

\begin{prop}
The map $f \mapsto \mathcal P_f^*$ is continuous at  every $f \in \Cex^*$.
\end{prop}
\begin{proof}
Let $\varepsilon > 0$ be small. Let $f^{(0)} = f$ and, recursively, 
$f^{(\ell + 1)} = (f^{(\ell)})_{\zeta_{f^{(\ell)}}(1-\varepsilon)}^*$. Define $s^{(\ell)} = \zeta_{f^{(\ell)}}(1-\varepsilon)$ and $\beta^{(\ell)} = \sum_{k=0}^\ell s^{(k)}$. 
Assume that $\|f_n-f\| \to 0$ for a sequence of continuous excursions $f_n, n \geq 1$. Denote by $r_n^{(0)}$ the sequence from Lemma \ref{lem_cont} with $r = \mathbf m_{f}(1-\varepsilon)$. Set $s_n^{(0)} := \zeta_{f_n}(r_n^{(0)})$ and 
$f_n^{(0)} := f_n$.
Then, for $\ell \geq 0$, 
inductively, $f^{(\ell+1)}_n := (f^{(\ell)}_n)_{s_n^{(\ell)}}^*$, where $s_n^{(\ell)} = \zeta_{f_n^{(\ell)}}(r_n^{(\ell)})$ and $r_n^{(\ell)}$ is the sequence from Lemma \ref{lem_cont} based on the functions $f^{(\ell)}, f^{(\ell)}_n, n \geq 1$ and 
$r =\mathbf m_{f^{(\ell)}}(1-\varepsilon)$. Let $\beta_n^{(\ell)} = \sum_{k=0}^\ell s_n^{(k)}$.

By Lemma \ref{lem_cont}, we have $s_n^{(\ell)} \to s^{(\ell)}$, hence $\beta_n^{(\ell)} \to \beta^{(\ell)}$  for all $\ell \geq 0$. Now, fix $K$ large and assume that $n$ is sufficiently large such that 
$s_n^{(\ell)} > 0$ for all $0 \leq \ell \leq K$.  Define $\lambda_n(\beta^{(\ell)}) = \beta_n^{(\ell)}$ for all $1 \leq \ell \leq K$, 
and linear on interval $[\beta^{(\ell)}, \beta^{(\ell+1)}]$, $0 \leq \ell \leq K-1$. Extend $\lambda_n$ to a continuous bijection on $[0,\infty)$ by a straight line of slope one for $t \geq \beta^{(K)}$. 
Clearly, $\lambda_n \to \text{id}$ uniformly on $[0,1]$. 
 Fix $\varepsilon' > 0$. By Lemma \ref{lem_cont}, for all $n$ sufficiently large, for all $1 \leq \ell \leq K$ we have $|d(\mathcal P^*_{f_n} (\beta_n^{(\ell)}), \mathcal P^*_f(\beta^{(\ell)}))| \leq \varepsilon'$ and
 $|d(\mathcal P^*_{f_n} (\beta_n^{(\ell)}-), \mathcal P^*_f(\beta^{(\ell)}-))| \leq \varepsilon'$. Hence, for those $n$,  
 $$\sup_{0 \leq t \leq \beta^{(K)}} |d(\mathcal P^*_{f_n} \circ \lambda_n (t), \mathcal P^*_f(t))| \leq \varepsilon + \varepsilon'.$$
 By construction, for those large $n$,  
$$\sup_{t > \beta^{(K)}} |d(\mathcal P^*_{f_n} \circ \lambda_n (t), \mathcal P^*_f(t))| \leq \omega_f((1-\varepsilon)^K) + \omega_{f_n}((1-\varepsilon)^K + 2 \varepsilon').$$
Since we choose both $\varepsilon, \varepsilon'$ arbitrarily small and $K$ arbitrarily large, this finishes the proof.
  \end{proof}

\begin{prop}
The map $f \mapsto \zeta_f$ is continuous at every $f \in \Cex^*$. (Here, $\zeta_f$ is considered as an element in the space of continuous functions on $[0,1]$ endowed with the supremum norm.) In particular, 
$f \mapsto T_\infty^f$ is continuous at $f \in \Cex^*$.
\end{prop}
\begin{proof}
First of all, it is easy to see that, for all $\varepsilon > 0$ and $f \in \Cex$, we have $\omega_{\zeta_f}(\varepsilon) \leq \omega_f(\varepsilon)$. Now, suppose that $\|f_n - f\| \to 0$ with $ \in \Cex^*$.
For ease of notation, let us abbreviate $\zeta_n := \zeta_{f_n}, n \geq 1$ and $\zeta = \zeta_f$. Clearly, $\zeta_n(1)=0$ for all $n$ sufficiently large. Hence, by the Arzela-Ascoli theorem, 
$(\zeta_n)$ is relatively compact. It suffices to prove that, for any $t \in (0,1)$, we have $\zeta_n(t) \to \zeta(t)$. 
Assume that $\mathbf m_f$ is continuous at $\zeta_f(t)$. Then $\mathbf m_{f_n}(\zeta_f(t)) \to \mathbf m_f(\zeta_f(t)) = t$. Clearly, for every $\varepsilon > 0$ we can find $t-\varepsilon  < s < t$ such that $\mathbf m_f$ is continuous at $\zeta_f(s)$. In particular, $\mathbf m_{f_n}(\zeta_f(s)) \to \mathbf m_f(\zeta_f(s)) = s$. This implies
$\liminf \zeta_n(t)  \geq \zeta(s)$. By continuity, it follows $\liminf \zeta_n(t)  \geq \zeta(t)$. Similarly, one shows $\limsup \zeta_n(t)  \leq \zeta(t)$.

Now assume that $\mathbf m_f$ is discontinuous at $\zeta_f(t)$. Then, there exist $t' \leq t \leq t''$ with $t' < t''$ such that $\zeta_f$ is constant on $[t', t'']$. For any $\varepsilon > 0$ there exists $t'-\varepsilon < s < t'$ such that $\mathbf m_f(s)$ is continuous at $s$. By the first part, this implies $\zeta_n(s) \to \zeta(s)$. By monotonicity,  $\limsup \zeta_n(t') \leq \zeta(s)$. By continuity, this implies $\limsup \zeta_n(t') \leq \zeta(t')$. Similarly, $\liminf \zeta_n(t'') \geq \zeta(t'')$. Since $\zeta(t') = \zeta(t'')$ this implies $\zeta_n(x) \to \zeta(x)$ for all $x \in [t', t'']$. 
\end{proof}

\textbf{Remark.} It is important to note that neither of the two propositions holds for general $f \in \Cex^{(1)}$ or $f \in \Cex^{(2)}$; both conditions are important.
 \medskip We can now apply the continuous mapping theorem. The following result contains the first statement in Theorem \ref{conv:hp}. Note that the quantity $T_\infty$ in Theorem \ref{conv:hp} equals $T_\infty^{\mathbf e}$ here.

\begin{thm} \label{thm:unif_conv_exc}
Let $\tau_n$ be a critical branching process with finite variance $\sigma^2$. 
\begin{itemize}
\item [(i)]
Let $L_n$ be the length of the corresponding heavy path. Then, in distribution,
$$\frac{L_n}{\sqrt n} \to \frac{2}{\sigma} \cdot T^{\mathbf e}_\infty\ .$$
\item [(ii)] For $k \geq 0$, let $P_n(k)$ be the size of the subtree rooted at the node on level $k$ on the heavy path. In distribution, in the Skorokhod topology on $\mathcal D_{[0, \infty)}$, 
\begin{align} \label{convP} \frac {P_{n}(\lfloor \cdot \ \sqrt{n}\rfloor)}{n} \to \mathbf m_{\frac{2}{\sigma} \mathbf e}\ .\end{align}
\item [(iii)]
 For $0 \leq \ell \leq n$, let $Q_n(\ell) = \inf \{ k \geq 0: P_n(k) \leq \ell \}$. Then, in distribution, on the space of continuous functions on $[0,1]$,
\begin{align} \label{convQ}  \frac{Q_n(\lfloor \cdot  \ n\rfloor) }{\sqrt n} \to \frac{2}{\sigma} \cdot \zeta_{\mathbf e}\ .\end{align}
\end{itemize}
\end{thm}

\medskip \noindent \textbf{The heavy path in the Brownian Continuum tree.}
Interval decompositions governed by a Brownian excursion can be studied with the help of self-similar fragmentations introduced by Bertoin \cite{bertoin_self}. 
We recall a version of Definition 2 in this work: a $\mathcal V$-valued process $F(t), t \geq 0$ with c{\`a}dl{\`a}g paths is called \emph{self-similar with index} $\alpha \in \R$, if 
\begin{enumerate} 
\item $F(0) = [0,1]$, $F(t) \subseteq F(s)$ for all $t \geq s \geq 0$; 
\item   $F(t)$ is continuous in probability at every $t \geq 0$; \end{enumerate}
further, given $F(t) = \cup I_j$ for $t \geq 0$ and disjoint open intervals $I_1, \ldots$,
\begin{enumerate} \setcounter{enumi}{2}
\item the processes $(F(t+s) \cap I_j)_{s \geq 0}, j \geq 1$ are stochastically independent;
\item for all $j \geq 1$, $F(t+s) \cap I_j, s \geq 0$ 
is distributed like $F(|I_j|^{\alpha} s), s \geq 0$ rescaled to fit on $I_j$. \end{enumerate}
Bertoin \cite{bertoin_self} observes that $\mathcal P_{\mathbf e}$ is a self-similar fragmentation process with $\alpha = -1/2$. Hence, the same follows for $\mathcal P^*_{\mathbf e}$. 
For $t \geq 0$, let $$\varrho_1(t) = \begin{cases} \inf \left \{u \geq 0: \int_0^u \sqrt{\mathbf m_{\mathbf e}(r)} dr > t \right \}, & \text{if } t < \int_0^\infty \sqrt{\mathbf m_{\mathbf e}(r)} dr \\ \infty & \text{otherwise}.\end{cases}$$

It follows from \cite[Theorem 2]{bertoin_self} that the $\mathcal V$-valued c{\`a}dl{\`a}g process $H(\cdot) := \mathcal P_{\mathbf e}(\varrho_1(\cdot))$ is a \emph{homogeneous} interval fragmentation, that is, a self-similar fragmentation process with index $\alpha = 0$. (Here, and subsequently, we abbreviate $\mathcal P_{\mathbf e}(\infty) = H(\infty) = \emptyset.$)
Homogeneous fragmentation processes were studied in detail in another work of Bertoin \cite{bertoin_hom}. In particular, by exploiting the connection between interval fragmentations and exchangeable partitions of the natural numbers \cite[Lemmas 5 and 6]{bertoin_self}, the arguments in the proof of
Theorem 3 in \cite{bertoin_hom} relying on a Poisson point process construction reveal that $\xi(\cdot) := -\log \lambda(H(\cdot))$ is a subordinator, that is, an increasing non-negative c{\`a}dl{\`a}g process with stationary and independent increments.
By \cite[Theorem 2]{bertoin_hom}, (the distribution) of a homogeneous fragmentation process is characterized by a unique exchangeable partition measure 
 which is determined by an \emph{erosion coefficient} $c \geq 0$ and a \emph{L{\'e}vy measure} $\nu$ on $(0,\infty)$ with the property that $\int_0^\infty \min(x,1) d\nu(x) <  \infty$. We refer to \cite{bertoin_hom} for a detailed discussion of this characterization and only use the following two results: 
first, by the arguments in \cite[Section 4]{bertoin_self}, for $\mathcal P^*_{\mathbf e}$, we have $c = 0$ and $$\nu(dx) = 2(2 \pi x^3(1-x)^3)^ {-1/2} \mathbf 1_{[1/2, 1)}(x) dx.$$ Second, by the arguments in the proof of Theorem 3 in 
\cite{bertoin_hom} the Laplace transform $\E{\exp(-q\xi(t))}, t, q \geq 0$ is given by $\exp(-t \Phi(q))$ with
 \begin{align} \label{def:phi} \Phi(q) = \int_{1/2}^1  (1-x^q)  d \nu (x). \end{align}
Summarizing, we obtain the following result, which is closely related to \cite[Corollary 2]{bertoin_self}.
\begin{prop}
 Let $\xi(t), t \geq 0$ be a subordinator with $\E{\exp(-q \xi(t))} = \exp(-t\Phi(q))$ as in \eqref{def:phi}. For $t \geq 0$, let
$$\varrho_2(t) = \begin{cases} \inf \left \{ u \geq 0 : \int_0^u e^{-\frac 1 2 \xi(r)} dr > t \right \}, & \text{if } t < \int_0^\infty e^{-\frac 1 2 \xi(r)} dr \\ \infty & \text{otherwise}.\end{cases}
$$
Then, $\exp(-\xi(\varrho_2(t))), t \geq 0$ and $\mathbf m_{\mathbf e}(t), t \geq 0$ are identically distributed.
\end{prop}

One can verify that
$$\Phi(q) = \frac{4}{\sqrt{\pi}}  \cdot {}_2F_1\left(-\frac 1 2, \frac 3 2 -q; \frac 1 2; \frac 1 2\right),$$
where ${}_2F_1$ denotes the standard hypergeometric function. In particular, 
$$\Phi \left( \frac 1 2 \right) = 2 \sqrt{\frac{2}{\pi}} \left(\sqrt{2} - \log (1+\sqrt{2}) \right).$$
Using recurrences for the hypergeometric function, one can also check that $\Phi(q) \sqrt{\pi}$ is rational for all $q = (2\ell + 1)/2$, $\ell \in \N, \ell \geq 1$.

As discussed in \cite[Section 4]{bertoin_self}, the work of Carmona, Petit and Yor \cite{carpetyor} yields explicit expressions for the moments of $T^{\mathbf e}_\infty$ which concludes the proof of the Theorem \ref{conv:hp}. See also Theorem 2 in Bertoin and Yor \cite{bertyor}.
\begin{prop}
$T^{\mathbf e}_\infty$ is distributed like $\int_0^\infty \exp(-\frac 1 2 \xi(t)) dt$. For $k \geq 1$, 
$$\E{(T^{\mathbf e}_\infty)^k} = \frac{k!}{\Phi(\frac 1 2) \cdots \Phi(\frac k 2)}.$$
\end{prop}

\medskip \noindent \textbf{A family of perpetuities.}
Let $0 < r < 1$. The dynamics of $\mathbf m_{\mathbf e}(t), t  \geq 0$ implies that
\begin{align} \label{perp} T^{\mathbf e}_\infty \stackrel{d}{=} \zeta_{\mathbf e}(r) + \sqrt{\mathbf m_{\mathbf e}(r)} T^{\mathbf e^*}_\infty,\end{align}
where $\mathbf e^*$ is an independent copy of $\mathbf e$. 
In particular, $T^{\mathbf e^*}_\infty, (\zeta_{\mathbf e}(r), \mathbf m_{\mathbf e}(r))$ are independent while 
$\zeta_{\mathbf e}(r), \mathbf m_{\mathbf e}(r)$ are defined using the \emph{same} Brownian excursion $\mathbf e$. 
Hence, $T^{\mathbf e}_\infty$ is characterized by a family of perpetuities, one for each value of $r$. For more background on stochastic fixed-point equations of perpetuity type and a proof for the fact that \eqref{perp} indeed determines the distribution of $T^{\mathbf e}_\infty$, we refer to Vervaat \cite{Vervaat1979a}.
For all $0 < r < 1$, stochastically, 
\begin{align} \label{boundperp} \sum_{k=0}^\infty \left( \frac{r}{2}\right) ^{k/2} \zeta_{\mathbf e}^{(k)}(r) \leq T^{\mathbf e}_\infty \leq \sum_{k=0}^\infty r^{k/2} \zeta_{\mathbf e}^{(k)}(r),\end{align}
where $\zeta_{\mathbf e}^{(0)}(r), \zeta_{\mathbf e}^{(1)}(r),  \ldots$ are independent copies of $\zeta_{\mathbf e}(r)$. 
Similarly, in the proof of Proposition \ref{heavyp}, we have shown that there exists a constant $C > 0$ and, for all $a > 2$ a constant $c > 0$ such that, stochastically,
\begin{align} \label{bounddis} c \sum_{k=1}^\infty a^{-k/2} |\mathcal N_k| \leq T^{\mathbf e}_\infty \leq C \sum_{k=0}^\infty 2^{-k/2} E_k, \end{align}
 where $\mathcal N_1, \mathcal N_2$ are independent standard normal random variables and $E_1, E_2, \ldots,$ are independent random variables with the standard exponential distribution. In fact, our proofs also revealed that, with the same constants $c, C, a$, in probability, 
 \begin{align} \label{boundcont} c  a^{-1/2} |\mathcal N_1| \leq \zeta_{\mathbf e}(r) \leq C  2^{-1/2} E_1. \end{align}
Note that the lower bound in \eqref{bounddis} does not follow from \eqref{boundperp} and \eqref{bounddis} due to the factor $1/2$ in \eqref{boundperp}. Hence, the tail bound deduced from the discrete-time approach is stronger than the bound we could show relying only on the perpetuity \eqref{perp}.

\end{document}